\documentclass[11pt,a4paper]{amsart}

\usepackage{amsmath,amssymb,amsthm,mathtools}
\usepackage{hyperref}
\usepackage{enumitem}
\usepackage{booktabs,array}

\hypersetup{
  colorlinks=true,
  linkcolor=blue,
  citecolor=blue,
  urlcolor=blue
}

\theoremstyle{plain}
\newtheorem{theorem}{Theorem}[section]
\newtheorem{proposition}[theorem]{Proposition}
\newtheorem{lemma}[theorem]{Lemma}
\newtheorem{corollary}[theorem]{Corollary}

\newtheorem{conjecture}[theorem]{Conjecture}

\theoremstyle{definition}
\newtheorem{definition}[theorem]{Definition}

\theoremstyle{remark}
\newtheorem{remark}[theorem]{Remark}
\newtheorem{problem}[theorem]{Problem}

\DeclareMathOperator{\sgn}{sgn}
\DeclareMathOperator{\Fix}{Fix}
\DeclareMathOperator{\std}{std}
\DeclareMathOperator{\perm}{perm}
\DeclareMathOperator{\triv}{triv}
\DeclareMathOperator{\ecc}{ecc}
\DeclareMathOperator{\Ind}{Ind}
\DeclareMathOperator{\Res}{Res}
\DeclareMathOperator{\diag}{diag}

\newcommand{\BN}{B_N}
\newcommand{\SN}{S_N}
\newcommand{\CC}{\mathbb{C}}
\newcommand{\EE}{\mathbb{E}}
\newcommand{\abs}[1]{\lvert #1 \rvert}
\newcommand{\ip}[2]{\langle #1,\, #2 \rangle}
\newcommand{\ket}[1]{\lvert #1 \rangle}
\newcommand{\bra}[1]{\langle #1 \rvert}

\newcommand{\Qsigned}{Q_{\mathrm{signed}}}
\newcommand{\Qdecomp}{Q_{\mathrm{decomp}}}
\newcommand{\QLV}{Q_{LV}}
\newcommand{\Rleq}[1]{R_{\leq #1}}
\newcommand{\hatBN}{\widehat{B_N}}

\begin{document}

\title[Quantum Query Complexity of $B_N$]{%
  Quantum Query Complexity of the Hyperoctahedral Group}

\author{JI HO BAE}
\address{[Address]}
\email{jihobae@snu.ac.kr}
\date{\today}

\subjclass[2020]{81P68, 20C30, 68Q12}
\keywords{quantum query complexity, hyperoctahedral group,
  oracle identification, representation theory, signed permutations}

\begin{abstract}
We determine the quantum query complexity of oracle identification
on the hyperoctahedral group
$\BN=\{\pm 1\}^N\rtimes\SN$ with respect to the natural
representation:
$\QLV(\BN)=2(N-1)$ for all $N\ge 2$.
This is twice the symmetric-group value
$\QLV(\SN)=N-1$~\cite{CopelandPommersheim}; the doubling arises
from an $\varepsilon$-parity obstruction that restricts the
bottleneck representation $\sgn(\sigma)$ to even tensor powers.
The proof combines a reduction to $\SN$ Kronecker products via
Rademacher moment polynomials with the bipartition distance formula
$d_T(((N),\varnothing),(\alpha,\beta))=2(N-\alpha_1)-\abs{\beta}$
in the tensor product graph.
A closed-form generating function yields the first-appearance
multiplicity $(2N-3)!!$.
We also show $\Qdecomp(\varphi)\le 2\,\Qsigned(\varphi)$, with
equality on~$B_2$, and conjecture a link between the adversary
bound and the graph eccentricity.
\end{abstract}

\maketitle
\setcounter{tocdepth}{2}
\tableofcontents

%% ================================================================
\section{Introduction}\label{sec:intro}
%% ================================================================

\subsection{Motivation}
The oracle identification problem asks: given black-box access to a
unitary representation $\rho(g)$ of a finite group $G$, how many
quantum queries suffice to determine the group element $g$?
Copeland and Pommersheim~\cite{CopelandPommersheim} introduced a
character-theoretic framework for this problem and determined the
query complexity for the symmetric group~$\SN$, building on earlier
quantum oracle models~\cite{BernsteinVazirani,vanDam98} and the
adversary method~\cite{Ambainis02,HoyerLeeSpalek,LeeMittalReichardtSpalekSzegedy}.
The representation theory of wreath products, of which the
hyperoctahedral group $\BN=\{\pm 1\}^N\rtimes\SN$ is the simplest
nonabelian case, is classical~\cite{JamesKerber,CeccheriniScarabottiTolli,GeckPfeiffer}.
We extend the Copeland--Pommersheim framework to~$\BN$.

\subsection{The hyperoctahedral group}
The \emph{hyperoctahedral group}
$\BN = \{\pm 1\}^N \rtimes \SN$
is the group of signed permutations on $N$ letters.  It is the
Weyl group of type $B_N$ (equivalently, type $C_N$) and acts as the
symmetry group of the $N$-dimensional hyperoctahedron
(cross-polytope).
Elements are pairs $(\sigma,\varepsilon)$ with $\sigma\in\SN$
and $\varepsilon=(\varepsilon_1,\dots,\varepsilon_N)\in\{\pm 1\}^N$,
and the group has order $\abs{\BN}=2^N\cdot N!$.
Its natural representation
$V=V_{((N-1),(1))}$ acts faithfully on $\CC^N$ by
$\rho(\sigma,\varepsilon)\ket{x}=\varepsilon_x\ket{\sigma(x)}$.
The oracle $U=P_\sigma D_\varepsilon$ factors into a permutation
matrix $P_\sigma$ and a diagonal sign matrix
$D_\varepsilon=\diag(\varepsilon_1,\dots,\varepsilon_N)$, giving rise
to two distinct oracle models.

\subsection{Main contributions}
Our three main results are as follows.

\begin{enumerate}[label=\textbf{(\Roman*)},leftmargin=*]
\item \textbf{Factor-$2$ simulation bound (Theorem~\ref{thm:simulation}).}
  For every decision problem $\varphi\colon\BN/\{\pm I\}\to\{0,1\}$,
  \[
    \Qdecomp(\varphi)\le 2\,\Qsigned(\varphi).
  \]
  The bound is tight: for the modified negative-fixed-point problem
  $\widetilde\varphi$ on $B_2$ we exhibit
  $\Qsigned(\widetilde\varphi)=1$ and
  $\Qdecomp(\widetilde\varphi)=2$.

\item \textbf{Oracle identification formula
  (Theorem~\ref{thm:main-QLV}).}
  For $N\ge 2$,
  \[
    \QLV(\BN)=2(N-1),
  \]
  with the sign-of-permutation representation
  $\sgn(\sigma)=V_{((1^N),\varnothing)}$ as the unique bottleneck
  for $N\ge 3$ (and one of four co-bottlenecks for $N=2$).

\item \textbf{$\gamma_2$-character bridge conjecture
  (Conjecture~\ref{conj:bridge}).}
  The adversary bound equals the eccentricity of $I(V)$ in the tensor
  product graph:
  $\gamma_{\mathrm{adv}}=\ecc_{T(\BN,V)}(I(V))=2N-3$.
\end{enumerate}

\subsection{Context}
The Copeland--Pommersheim framework~\cite{CopelandPommersheim}
shows that nonadaptive algorithms are optimal for oracle
identification with group representations and expresses the optimal
success probability via the cumulative spectral reach.
Passing from $\SN$ to $\BN$ introduces new representation-theoretic
phenomena --- notably the $\varepsilon$-parity obstruction and the
resulting double-step structure --- which are the focus of this
paper.

%% ================================================================
\section{Preliminaries}\label{sec:prelim}
%% ================================================================

\subsection{The hyperoctahedral group \texorpdfstring{$\BN$}{B\_N}}

\begin{definition}\label{def:BN}
The \emph{hyperoctahedral group} $\BN=\{\pm 1\}^N\rtimes \SN$
consists of pairs $(\sigma,\varepsilon)$ with $\sigma\in\SN$ and
$\varepsilon=(\varepsilon_1,\dots,\varepsilon_N)\in\{\pm 1\}^N$.
The group law is
\begin{equation}\label{eq:group-law}
  (\sigma,\varepsilon)\cdot(\tau,\delta)
  =\bigl(\sigma\tau,\;
    (\varepsilon_{\tau(1)}\delta_1,\;
     \varepsilon_{\tau(2)}\delta_2,\;
     \dots,\;
     \varepsilon_{\tau(N)}\delta_N)\bigr).
\end{equation}
The group has order $\abs{\BN}=2^N\cdot N!$.
\end{definition}

\subsection{Irreducible representations}

\begin{definition}\label{def:irreps}
The irreducible representations of $\BN$ are indexed by
\emph{bipartitions} $(\alpha,\beta)\vdash N$, i.e., pairs of
partitions $\alpha,\beta$ with $\abs{\alpha}+\abs{\beta}=N$.
We write $V_{(\alpha,\beta)}$ for the corresponding irreducible
representation.
\end{definition}

\subsection{The natural representation}

\begin{definition}\label{def:nat-rep}
The \emph{natural representation}
$V=V_{((N-1),(1))}$ of $\BN$ acts on $\CC^N$ by
\begin{equation}\label{eq:nat-rep}
  \rho(\sigma,\varepsilon)\ket{x}
  =\varepsilon_x\,\ket{\sigma(x)},
  \qquad x\in\{1,\dots,N\}.
\end{equation}
In matrix form, $\rho(\sigma,\varepsilon)=P_\sigma D_\varepsilon$,
where $P_\sigma$ is the $N\times N$ permutation matrix of $\sigma$
and $D_\varepsilon=\diag(\varepsilon_1,\dots,\varepsilon_N)$ is the
diagonal sign matrix.
\end{definition}

\begin{proposition}[Character formula]\label{prop:char-V}
The character of the natural representation is
\begin{equation}\label{eq:char-V}
  \chi_V(\sigma,\varepsilon)
  =\sum_{\substack{x\colon\sigma(x)=x}}\varepsilon_x.
\end{equation}
\end{proposition}

\begin{proof}
The diagonal entries of $P_\sigma D_\varepsilon$ are
$(P_\sigma D_\varepsilon)_{xx}=\delta_{\sigma(x),x}\,\varepsilon_x$.
Summing over $x$ gives~\eqref{eq:char-V}.
\end{proof}

\begin{proposition}[Faithfulness]\label{prop:faithful}
The natural representation $\rho$ is faithful:
$\ker(\rho)=\{e\}$.
\end{proposition}

\begin{proof}
If $P_\sigma D_\varepsilon=I_N$, then for every $x$,
$\varepsilon_x\ket{\sigma(x)}=\ket{x}$.
This forces $\sigma(x)=x$ for all $x$ and $\varepsilon_x=+1$ for
all $x$, hence $(\sigma,\varepsilon)$ is the identity element.
\end{proof}

\subsection{Oracle models}

\begin{definition}[Signed oracle model]\label{def:signed-oracle}
The algorithm has black-box access to the unitary
$U=P_\sigma D_\varepsilon$ acting on the query register $\CC^N$.
Each call to $U$ (or $U^\dagger$) counts as one query.
Between queries, the algorithm may apply arbitrary unitaries on the
query register tensored with an ancilla register of arbitrary
dimension.  We write $\Qsigned(\varphi)$ for the minimum number of
queries needed to compute a decision problem $\varphi$ with certainty.
\end{definition}

\begin{definition}[Decomposed oracle model]\label{def:decomp-oracle}
The algorithm has separate black-box access to $P_\sigma$ and
$D_\varepsilon$, each acting on $\CC^N$.  Each call to either oracle
counts as one query.
The adjoint oracles $P_\sigma^\dagger = P_{\sigma^{-1}}$ and
$D_\varepsilon^\dagger = D_\varepsilon$ are also available at unit
cost.
The total query count is $t_P+t_D$, and the
algorithm may choose classically which oracle to call at each step.
We write $\Qdecomp(\varphi)$ for the minimum total queries needed.
\end{definition}

\subsection{The Copeland--Pommersheim framework}

\begin{definition}[Oracle identification]\label{def:oracle-id}
Given a group $G$ and a faithful unitary representation
$\rho\colon G\to U(\mathcal{H})$, the \emph{oracle identification
problem} asks to determine $g\in G$ from queries to the oracle
$\rho(g)$.
\end{definition}

\begin{definition}\label{def:spectral-reach}
For a representation $\rho$ and integer $t\ge 1$, let
$I(\rho^{\otimes t})$ denote the set of irreducible representations
appearing as constituents of $\rho^{\otimes t}$.  The
\emph{cumulative spectral reach} is
$\Rleq{t}(\rho)=\bigcup_{s=1}^{t}I(\rho^{\otimes s})$.
\end{definition}

\begin{theorem}[Copeland--Pommersheim]\label{thm:CP}
The optimal success probability after $t$~queries to $\rho(g)$ is
\begin{equation}\label{eq:CP}
  P_{\mathrm{opt}}^{(t)}
  =\frac{1}{\abs{G}}
    \sum_{\pi\in\Rleq{t}(\rho)}(\dim\pi)^2,
\end{equation}
and nonadaptive algorithms are optimal.
\end{theorem}

\begin{definition}[Tensor product graph]\label{def:TPgraph}
Assume $\rho\cong\rho^\ast$ (i.e.\ $\rho$ is self-dual).
The \emph{tensor product graph} $T(G,\rho)$ has vertex set
$\widehat{G}$ (the set of all irreducible representations of $G$)
and an edge $\pi\sim\tau$ whenever
$\tau\in I(\rho\otimes\pi)$.
(Self-duality ensures symmetry: $\tau\in I(\rho\otimes\pi)$ iff
$\pi\in I(\rho\otimes\tau)$, so the graph is undirected.)
\end{definition}

\begin{remark}\label{rmk:self-dual}
The natural representation $V=V_{((N-1),(1))}$ of $\BN$ is real
(its character is real-valued), hence self-dual.  All later
applications of $T(G,\rho)$ use $\rho=V$.
\end{remark}

\begin{proposition}\label{prop:ecc-formula}
Under the self-duality hypothesis of
Definition~\textup{\ref{def:TPgraph}}, the query complexity
satisfies
\begin{equation}\label{eq:ecc}
  \QLV(G,\rho)=\ecc_{T(G,\rho)}\bigl(I(\rho)\bigr)+1,
\end{equation}
where
$\ecc(S)=\max_{\pi\in\widehat{G}}\,d_T(\pi,S)$ is the
eccentricity of the set $I(\rho)$ in the tensor product graph.
\end{proposition}

\begin{proof}
The cumulative reach $\Rleq{t}$ captures all irreducible
representations at graph distance $\le t-1$ from $I(\rho)$: the
first tensor power gives $I(\rho)$ (distance $0$), the second adds
neighbours (distance $1$), and so on.  Full identification requires
$\Rleq{t}=\widehat{G}$, which occurs at $t=\ecc(I(\rho))+1$.
\end{proof}

%% ================================================================
\section{Oracle Models and the Factor-\texorpdfstring{$2$}{2}
  Separation}\label{sec:factor2}
%% ================================================================

\subsection{The global phase obstruction}

\begin{proposition}[Global phase indistinguishability]
\label{prop:global-phase}
For any quantum algorithm making $t$~queries to the oracle
$U_{(\sigma,\varepsilon)}$, the output distribution is identical on
input $(\sigma,\varepsilon)$ and $(\sigma,-\varepsilon)$.
\end{proposition}

\begin{proof}
Let the algorithm produce the final state
\[
  \ket{\psi_{\mathrm{final}}}
  = V_t(U\otimes I_A)\,V_{t-1}\cdots V_1(U\otimes I_A)\,V_0\ket{\psi_0},
\]
where $V_0,\dots,V_t$ are query-independent unitaries and
$U=P_\sigma D_\varepsilon$ acts on the query register.
Since $\rho(\sigma,-\varepsilon)=-\rho(\sigma,\varepsilon)$,
replacing $U$ by $-U$ multiplies each oracle application by $-1$,
yielding $(-1)^t\ket{\psi_{\mathrm{final}}}$.
The same holds for $U^\dagger$ queries: since
$(-U)^\dagger = -U^\dagger$, each adjoint query also acquires a
factor of~$-1$.
This global phase does not affect measurement probabilities.
\end{proof}

\begin{corollary}\label{cor:phase-constraint}
Any decision problem
$\varphi\colon\BN\to\{0,1\}$ computable from the signed oracle must
satisfy
$\varphi(\sigma,\varepsilon)=\varphi(\sigma,-\varepsilon)$
for all $(\sigma,\varepsilon)\in\BN$, i.e.,
$\varphi$ is well-defined on the quotient $\BN/\{\pm I\}$.
\end{corollary}

\subsection{The modified negative-fixed-point problem}

\begin{definition}[Modified NFP]\label{def:mod-NFP}
The \emph{modified negative-fixed-point} function
$\widetilde{\varphi}\colon\BN\to\{0,1\}$ is
\[
  \widetilde{\varphi}(\sigma,\varepsilon)
  =\begin{cases}
    1 & \text{if } \sigma=\mathrm{id} \text{ and }
        \varepsilon\notin\{(+1,\dots,+1),(-1,\dots,-1)\},\\
    0 & \text{otherwise}.
  \end{cases}
\]
\end{definition}

\begin{proposition}[Well-definedness]\label{prop:NFP-welldef}
The function $\widetilde{\varphi}$ is constant on global phase
equivalence classes
$\{(\sigma,\varepsilon),(\sigma,-\varepsilon)\}$ for all $N\ge 2$.
\end{proposition}

\begin{proof}
If $\sigma\ne\mathrm{id}$, then both elements map to~$0$.  If
$\sigma=\mathrm{id}$ and $\varepsilon=(+1,\dots,+1)$, then
$-\varepsilon=(-1,\dots,-1)$; both are ``all-equal'' and map to~$0$.
If $\sigma=\mathrm{id}$ and $\varepsilon$ has mixed signs, then
$-\varepsilon$ also has mixed signs (since $N\ge 2$), and both map
to~$1$.
\end{proof}

\begin{remark}[Why not parity]\label{rmk:not-parity}
An alternative using parity of the number of negative signs
fails for odd $N$: if $\varepsilon$ has $k$ negatives, then
$-\varepsilon$ has $N-k$ negatives, and $k+(N-k)=N$ is odd, so
$k$ and $N-k$ have different parities.
The ``mixed signs'' definition works uniformly for all~$N\ge 2$.
\end{remark}

\subsection{The factor-\texorpdfstring{$2$}{2} simulation theorem}

\begin{theorem}[Factor-$2$ simulation]\label{thm:simulation}
For any decision problem
$\varphi\colon\BN/\{\pm I\}\to\{0,1\}$,
\begin{equation}\label{eq:sim-bound}
  \Qdecomp(\varphi)\le 2\,\Qsigned(\varphi).
\end{equation}
\end{theorem}

\begin{proof}
Let $\mathcal{A}$ be a $t$-query signed oracle algorithm producing
\[
  \ket{\psi_{\mathrm{final}}}
  = V_t\,(U\otimes I_A)\,V_{t-1}\cdots
    V_1\,(U\otimes I_A)\,V_0\,\ket{\psi_0}.
\]
Replace each occurrence of $U\otimes I_A$ by the product
$(P_\sigma\otimes I_A)(D_\varepsilon\otimes I_A)$.
Since $P_\sigma D_\varepsilon=U$ exactly, the resulting state
\[
  \ket{\psi'_{\mathrm{final}}}
  = V_t\,(P_\sigma\otimes I_A)(D_\varepsilon\otimes I_A)\,
    V_{t-1}\cdots
    V_1\,(P_\sigma\otimes I_A)(D_\varepsilon\otimes I_A)\,
    V_0\,\ket{\psi_0}
\]
equals $\ket{\psi_{\mathrm{final}}}$.  The simulation is exact.
Each of the $t$ signed-oracle calls is replaced by one call to
$D_\varepsilon$ and one call to $P_\sigma$, for a total of $2t$
decomposed queries.
For $U^\dagger = D_\varepsilon P_{\sigma^{-1}}$ queries, replace
each $U^\dagger\otimes I_A$ by
$(D_\varepsilon\otimes I_A)(P_{\sigma^{-1}}\otimes I_A)$, again
costing $2$ decomposed queries.
\end{proof}

\begin{corollary}[No super-constant separation]\label{cor:no-poly-sep}
No decision problem on $\BN$ admits $\Qsigned=O(1)$ and
$\Qdecomp=\omega(1)$.  In particular, no polynomial separation is
possible between the two models.
\end{corollary}

\begin{proof}
If $\Qsigned(\varphi)\le C$ for a constant $C$, then
$\Qdecomp(\varphi)\le 2C=O(1)$ by Theorem~\ref{thm:simulation}.
\end{proof}

\begin{remark}[Asymmetry]\label{rmk:asymmetry}
The reverse simulation --- using the signed oracle to simulate
individual decomposed queries --- is not straightforward.  Simulating
a query to $D_\varepsilon$ alone requires computing
$D_\varepsilon=UP_\sigma^{-1}$, which needs the inverse permutation
oracle.  Whether a general bound
$\Qsigned\le C\cdot\Qdecomp$ exists remains open.
\end{remark}

\subsection{Tightness for \texorpdfstring{$B_2$}{B\_2}:
  signed oracle upper bound}

We exhibit a $1$-query signed oracle algorithm for
$\widetilde{\varphi}$ on~$B_2$.

\begin{lemma}[Key identity]\label{lem:bell-identity}
For any $2\times 2$ matrix
$A=\bigl(\begin{smallmatrix}a&b\\c&d\end{smallmatrix}\bigr)$,
\begin{equation}\label{eq:bell-key}
  \bra{\Phi^-}(A\otimes I)\ket{\Phi^+}=\frac{a-d}{2},
\end{equation}
where
$\ket{\Phi^+}=\tfrac{1}{\sqrt{2}}(\ket{00}+\ket{11})$ and
$\ket{\Phi^-}=\tfrac{1}{\sqrt{2}}(\ket{00}-\ket{11})$.
\end{lemma}

\begin{proof}
Compute
\[
  (A\otimes I)\ket{\Phi^+}
  =\frac{1}{\sqrt{2}}
    \bigl(a\ket{00}+c\ket{10}+b\ket{01}+d\ket{11}\bigr).
\]
Taking the inner product with
$\ket{\Phi^-}=\frac{1}{\sqrt{2}}(\ket{00}-\ket{11})$ yields
$(a-d)/2$.
\end{proof}

\begin{theorem}[Signed oracle complexity of $\widetilde\varphi$ on
  $B_2$]\label{thm:B2-signed}
$\Qsigned(\widetilde{\varphi})=1$.
\end{theorem}

\begin{proof}
\textbf{Algorithm.}
(1)~Prepare $\ket{\Phi^+}$ on the query register $\CC^2$ and an
ancilla $\CC^2$.
(2)~Apply $U\otimes I$ (one oracle query).
(3)~Measure the projector
$\ket{\Phi^-}\!\bra{\Phi^-}$.
(4)~Output $\widetilde\varphi=1$ if outcome is
$\ket{\Phi^-}$; output $0$ otherwise.

\textbf{Correctness.}
By Lemma~\ref{lem:bell-identity} with $A=U=P_\sigma D_\varepsilon$,
the probability of measuring $\ket{\Phi^-}$ is
$\lvert(a-d)/2\rvert^2$.
We verify this equals $\widetilde\varphi$ for every element of
$B_2$ in Table~\ref{tab:B2-verify}.

\begin{table}[ht]
\centering
\caption{Verification of the $1$-query signed oracle algorithm on
$B_2$.}\label{tab:B2-verify}
\begin{tabular}{@{}lcccccccc@{}}
\toprule
Element & $\sigma$ & $\varepsilon$ & $U$ & $a$ & $d$
  & $\frac{a-d}{2}$ & $\bigl\lvert\frac{a-d}{2}\bigr\rvert^2$
  & $\widetilde\varphi$ \\
\midrule
$(\mathrm{id},+,+)$   & $\mathrm{id}$ & $(+,+)$ & $I$
  & $1$  & $1$  & $0$  & $0$ & $0$ \\
$(\mathrm{id},+,-)$   & $\mathrm{id}$ & $(+,-)$ & $\diag(1,-1)$
  & $1$  & $-1$ & $1$  & $1$ & $1$ \\
$(\mathrm{id},-,+)$   & $\mathrm{id}$ & $(-,+)$ & $\diag(-1,1)$
  & $-1$ & $1$  & $-1$ & $1$ & $1$ \\
$(\mathrm{id},-,-)$   & $\mathrm{id}$ & $(-,-)$ & $-I$
  & $-1$ & $-1$ & $0$  & $0$ & $0$ \\
$((12),+,+)$          & $(12)$ & $(+,+)$ & $X$
  & $0$  & $0$  & $0$  & $0$ & $0$ \\
$((12),+,-)$          & $(12)$ & $(+,-)$
  & $\bigl(\begin{smallmatrix}0&-1\\1&\phantom{-}0\end{smallmatrix}\bigr)$
  & $0$  & $0$  & $0$  & $0$ & $0$ \\
$((12),-,+)$          & $(12)$ & $(-,+)$
  & $\bigl(\begin{smallmatrix}\phantom{-}0&1\\-1&0\end{smallmatrix}\bigr)$
  & $0$  & $0$  & $0$  & $0$ & $0$ \\
$((12),-,-)$          & $(12)$ & $(-,-)$
  & $\bigl(\begin{smallmatrix}0&-1\\-1&\phantom{-}0\end{smallmatrix}\bigr)$
  & $0$  & $0$  & $0$  & $0$ & $0$ \\
\bottomrule
\end{tabular}
\end{table}

In every row the algorithm output matches $\widetilde\varphi$,
so $\Qsigned(\widetilde\varphi)\le 1$.  Since
$\widetilde\varphi$ is non-constant, $\Qsigned\ge 1$, whence
$\Qsigned(\widetilde\varphi)=1$.
\end{proof}

\subsection{Tightness for \texorpdfstring{$B_2$}{B\_2}:
  decomposed oracle lower bound}

\begin{theorem}[Decomposed oracle lower bound for $B_2$]
\label{thm:B2-decomp-lb}
$\Qdecomp(\widetilde\varphi)\ge 2$.
\end{theorem}

\begin{proof}
A $1$-query decomposed algorithm queries either $P_\sigma$ or
$D_\varepsilon$ (not both).

\textbf{Case 1: one query to $P_\sigma$.}
The elements
$g_1=(\mathrm{id},+,+)$ and $g_2=(\mathrm{id},+,-)$
share $P_{\mathrm{id}}=I$, making them indistinguishable.
Yet $\widetilde\varphi(g_1)=0\ne 1=\widetilde\varphi(g_2)$.

\textbf{Case 2: one query to $D_\varepsilon$.}
The elements $g_2=(\mathrm{id},+,-)$ and
$g_6=((12),+,-)$ share
$D_{(+,-)}=\diag(1,-1)$, making them indistinguishable.
Yet $\widetilde\varphi(g_2)=1\ne 0=\widetilde\varphi(g_6)$.

In both cases the algorithm faces a pair of inputs with identical
oracle operators but differing required outputs.
Therefore $\Qdecomp(\widetilde\varphi)\ge 2$.
\end{proof}

\subsection{Tightness for \texorpdfstring{$B_2$}{B\_2}:
  decomposed oracle upper bound}

\begin{theorem}[Decomposed oracle complexity of $\widetilde\varphi$
  on $B_2$]\label{thm:B2-decomp}
$\Qdecomp(\widetilde\varphi)=2$.
\end{theorem}

\begin{proof}
It suffices to exhibit a $2$-query decomposed algorithm, since the
lower bound $\Qdecomp\ge 2$ is given by
Theorem~\ref{thm:B2-decomp-lb}.

\textbf{Algorithm.}
\begin{enumerate}[label=(\arabic*)]
\item \textbf{Query $D_\varepsilon$:}
  Prepare $\ket{+}=\frac{1}{\sqrt{2}}(\ket{0}+\ket{1})$.
  Apply $D_\varepsilon$.  Apply the Hadamard gate $H$.
  Measure in the computational basis; let $b_1$ be the outcome.
\item \textbf{Query $P_\sigma$:}
  Prepare $\ket{1}$ (the first coordinate basis state).
  Apply $P_\sigma$.
  Measure in the coordinate basis $\{\ket{1},\ket{2}\}$;
  set $b_2=0$ for outcome $\ket{1}$ and $b_2=1$ for outcome
  $\ket{2}$.
\item \textbf{Output:}
  $\widetilde\varphi=b_1\wedge\overline{b_2}$.
\end{enumerate}

\textbf{Analysis of Query~1.}
$D_\varepsilon\ket{+}
=\frac{1}{\sqrt{2}}(\varepsilon_1\ket{0}+\varepsilon_2\ket{1})$.
After the Hadamard:
\begin{itemize}
\item $\varepsilon=(+,+)$:
  $D_\varepsilon\ket{+}=\ket{+}$,
  $H\ket{+}=\ket{0}$, so $b_1=0$.
\item $\varepsilon=(-,-)$:
  $D_\varepsilon\ket{+}=-\ket{+}$,
  $H(-\ket{+})=-\ket{0}$, so $b_1=0$.
\item $\varepsilon=(+,-)$:
  $D_\varepsilon\ket{+}=\ket{-}$,
  $H\ket{-}=\ket{1}$, so $b_1=1$.
\item $\varepsilon=(-,+)$:
  $D_\varepsilon\ket{+}=-\ket{-}$,
  $H(-\ket{-})=-\ket{1}$, so $b_1=1$.
\end{itemize}
Thus $b_1=1$ if and only if $\varepsilon$ has mixed signs.

\textbf{Analysis of Query~2.}
We work in the coordinate basis $\{\ket{1},\ket{2}\}$ of~$\CC^2$
(matching the group action $\sigma\colon\{1,2\}\to\{1,2\}$).
Prepare $\ket{1}$ and apply $P_\sigma$: the outcome is
$P_\sigma\ket{1}=\ket{\sigma(1)}$.
For $\sigma=\mathrm{id}$: outcome $\ket{1}$, set $b_2=0$.
For $\sigma=(12)$: outcome $\ket{2}$, set $b_2=1$.

\textbf{Decision.}
$b_1\wedge\overline{b_2}=1$ if and only if $\varepsilon$ has mixed
signs \emph{and} $\sigma=\mathrm{id}$, which is precisely
$\widetilde\varphi=1$.
\end{proof}

\subsection{Structural necessity}

\begin{theorem}[Individual oracles are insufficient]
\label{thm:structural-necessity}
Neither $P_\sigma$ alone nor $D_\varepsilon$ alone can compute
$\widetilde\varphi$ on $B_2$ in any finite number of queries.
\end{theorem}

\begin{proof}
\textbf{$P_\sigma$ alone:}
For any number of queries $t$, elements
$g_1=(\mathrm{id},+,+)$ and $g_2=(\mathrm{id},+,-)$
yield the same oracle $P_{\mathrm{id}}=I$ at every query, hence
produce identical final states.  Since
$\widetilde\varphi(g_1)\ne\widetilde\varphi(g_2)$,
no finite~$t$ suffices.

\textbf{$D_\varepsilon$ alone:}
Elements $g_2=(\mathrm{id},+,-)$ and $g_6=((12),+,-)$
yield the same oracle $D_{(+,-)}=\diag(1,-1)$ at every query.
Since $\widetilde\varphi(g_2)\ne\widetilde\varphi(g_6)$,
no finite~$t$ suffices.
\end{proof}

%% ================================================================
\section{The Reduction Formula}\label{sec:reduction}
%% ================================================================

\subsection{The \texorpdfstring{$\varepsilon$}{epsilon}-parity
  selection rule}

\begin{proposition}[Parity of the natural character]
\label{prop:char-parity}
For all $(\sigma,\varepsilon)\in\BN$,
\begin{equation}\label{eq:char-parity}
  \chi_V(\sigma,-\varepsilon)=-\chi_V(\sigma,\varepsilon).
\end{equation}
\end{proposition}

\begin{proof}
$\chi_V(\sigma,\varepsilon)=\sum_{x\colon\sigma(x)=x}\varepsilon_x$.
Replacing $\varepsilon$ by $-\varepsilon$ negates each summand.
\end{proof}

\begin{corollary}\label{cor:parity-tensor}
$\chi_V^t(\sigma,-\varepsilon)=(-1)^t\chi_V^t(\sigma,\varepsilon)$.
\end{corollary}

\begin{definition}[$\varepsilon$-parity of irreps]
\label{def:eps-parity}
An irreducible representation $V_{(\alpha,\beta)}$ of $\BN$ is
\emph{$\varepsilon$-even} if
$\chi_{(\alpha,\beta)}(\sigma,-\varepsilon)
=\chi_{(\alpha,\beta)}(\sigma,\varepsilon)$
for all $(\sigma,\varepsilon)$, and \emph{$\varepsilon$-odd} if
$\chi_{(\alpha,\beta)}(\sigma,-\varepsilon)
=-\chi_{(\alpha,\beta)}(\sigma,\varepsilon)$
for all $(\sigma,\varepsilon)$.
The $\varepsilon$-parity of $V_{(\alpha,\beta)}$ is $(-1)^{\abs{\beta}}$
(since the central element $(\mathrm{id},-1,\dots,-1)\in\BN$ acts on
$V_{(\alpha,\beta)}$ by the scalar $(-1)^{\abs{\beta}}$).
\end{definition}

\begin{proposition}[Parity selection rule]
\label{prop:parity-selection}
The irreducible constituents of $V^{\otimes t}$ satisfy:
\begin{enumerate}[label=\textup{(\roman*)}]
\item if $t$ is even, all constituents are $\varepsilon$-even;
\item if $t$ is odd, all constituents are $\varepsilon$-odd.
\end{enumerate}
In particular, $\sgn(\sigma)=V_{((1^N),\varnothing)}$ is
$\varepsilon$-even and can only appear in even tensor powers.
\end{proposition}

\begin{proof}
Suppose $W$ appears in $V^{\otimes t}$, so
$\ip{\chi_V^t}{\chi_W}_{B_N}\ne 0$.
For each $(\sigma,\varepsilon)\in B_N$, pair it with
$(\sigma,-\varepsilon)$.  By
Corollary~\ref{cor:parity-tensor},
$\chi_V^t(\sigma,-\varepsilon)=(-1)^t\chi_V^t(\sigma,\varepsilon)$.
If $W$ has parity $(-1)^{t+1}$ (the ``wrong'' parity), then
the contributions from $(\sigma,\varepsilon)$ and
$(\sigma,-\varepsilon)$ cancel, giving
$\ip{\chi_V^t}{\chi_W}=0$, a contradiction.
\end{proof}

\begin{lemma}[Parity obstruction for odd tensor powers]
\label{lem:parity-odd}
For any odd positive integer $t$,
\begin{equation}\label{eq:parity-odd}
  \ip{\chi_V^t}{\chi_{\sgn}}_{B_N}=0.
\end{equation}
\end{lemma}

\begin{proof}
For odd $t$,
$\chi_V^t(\sigma,-\varepsilon)=-\chi_V^t(\sigma,\varepsilon)$ while
$\chi_{\sgn}(\sigma,-\varepsilon)=\sgn(\sigma)
=\chi_{\sgn}(\sigma,\varepsilon)$.
Grouping terms $(\sigma,\varepsilon)$ and $(\sigma,-\varepsilon)$
in the inner product sum:
\begin{align*}
  \ip{\chi_V^t}{\chi_{\sgn}}
  &= \frac{1}{\abs{B_N}}
     \sum_{\sigma}\sgn(\sigma)
     \sum_{\substack{\varepsilon\\ \varepsilon_1=+1}}
     \bigl[\chi_V^t(\sigma,\varepsilon)
       +\chi_V^t(\sigma,-\varepsilon)\bigr]\\
  &= \frac{1}{\abs{B_N}}
     \sum_{\sigma}\sgn(\sigma)
     \sum_{\substack{\varepsilon\\ \varepsilon_1=+1}}
     \bigl[\chi_V^t(\sigma,\varepsilon)
       -\chi_V^t(\sigma,\varepsilon)\bigr]
  =0.\qedhere
\end{align*}
\end{proof}

\subsection{The reduction formula}

\begin{theorem}[Reduction to $\SN$]\label{thm:reduction}
For any positive integer $m$,
\begin{equation}\label{eq:reduction}
  \ip{\chi_V^{2m}}{\sgn(\sigma)}_{B_N}
  =\frac{1}{N!}\sum_{\sigma\in\SN}
    \sgn(\sigma)\cdot p_m\bigl(\abs{\Fix(\sigma)}\bigr),
\end{equation}
where $f(\sigma)=\abs{\Fix(\sigma)}$ and the \emph{Rademacher moment
polynomial} is
\begin{equation}\label{eq:Rademacher}
  p_m(f)
  =\EE\biggl[\Bigl(\sum_{i=1}^{f}X_i\Bigr)^{\!2m}\biggr],
  \qquad X_1,\dots,X_f
    \;\text{iid Rademacher}\;
    (\Pr[X_i=\pm 1]=\tfrac{1}{2}).
\end{equation}
\end{theorem}

\begin{proof}
By definition of the inner product over $\BN$,
\[
  \ip{\chi_V^{2m}}{\sgn(\sigma)}_{B_N}
  =\frac{1}{2^N N!}
    \sum_{(\sigma,\varepsilon)\in B_N}
      \chi_V(\sigma,\varepsilon)^{2m}\cdot\sgn(\sigma).
\]
For fixed $\sigma$, write $f=\abs{\Fix(\sigma)}$.  Then
$\chi_V(\sigma,\varepsilon)
=\sum_{x\in\Fix(\sigma)}\varepsilon_x$
depends only on the signs at fixed points.
The sum over $\varepsilon\in\{\pm 1\}^N$ factorises: the
$N-f$ coordinates at non-fixed points contribute freely (a factor of
$2^{N-f}$), and the $f$ fixed-point coordinates contribute
\[
  \sum_{(\varepsilon_x)_{x\in\Fix(\sigma)}\in\{\pm 1\}^f}
  \Bigl(\sum_{x\in\Fix(\sigma)}\varepsilon_x\Bigr)^{\!2m}
  =2^f\cdot p_m(f).
\]
Combining,
$\sum_{\varepsilon\in\{\pm 1\}^N}\chi_V(\sigma,\varepsilon)^{2m}
=2^{N-f}\cdot 2^f\cdot p_m(f)
=2^N\cdot p_m(f)$.
Substituting back:
\[
  \ip{\chi_V^{2m}}{\sgn(\sigma)}_{B_N}
  =\frac{1}{2^N N!}\sum_{\sigma\in\SN}
    \sgn(\sigma)\cdot 2^N\cdot p_m(f(\sigma))
  =\frac{1}{N!}\sum_{\sigma\in\SN}
    \sgn(\sigma)\cdot p_m(f(\sigma)).\qedhere
\]
\end{proof}

\begin{remark}\label{rmk:reduction-significance}
The Reduction Formula transforms a computation over $\BN$ (a group of
order $2^N N!$) into one over $\SN$ (order $N!$), at the cost of
introducing the moment polynomial $p_m$.  This is possible because
the sum over signs $\varepsilon$ factorises from the sum over
permutations $\sigma$.
\end{remark}

\subsection{The degree of \texorpdfstring{$p_m$}{p\_m}}

\begin{theorem}[Degree of the moment polynomial]
\label{thm:degree-pm}
The polynomial $p_m(f)$ has degree exactly $m$ in $f$.
More precisely,
\begin{equation}\label{eq:pm-binom}
  p_m(f)=\sum_{j=1}^{m}S^{\mathrm{even}}(2m,j)\cdot\binom{f}{j},
\end{equation}
where $S^{\mathrm{even}}(2m,j)$ denotes the number of surjections
from $\{1,\dots,2m\}$ to a $j$-element set in which every fibre has
even cardinality.  Moreover, $S^{\mathrm{even}}(2m,j)>0$ for all
$1\le j\le m$ and $S^{\mathrm{even}}(2m,j)=0$ for $j>m$.
\end{theorem}

\begin{proof}
Expand the $2m$-th power:
\[
  p_m(f)
  =\sum_{(i_1,\dots,i_{2m})\in[f]^{2m}}
    \EE[X_{i_1}\cdots X_{i_{2m}}].
\]
The expectation $\EE[X_{i_1}\cdots X_{i_{2m}}]$ equals $1$ if and
only if every distinct index among $\{i_1,\dots,i_{2m}\}$ appears an
even number of times, and $0$ otherwise (since
$\EE[X^{2k}]=1$ and $\EE[X^{2k+1}]=0$ for iid Rademacher variables).

Classify tuples by the number $j$ of distinct indices.  Each index
must appear at least twice, so $2j\le 2m$, giving $j\le m$.  For
$j>m$ no valid tuples exist.

For a given $j$-element subset $T\subseteq[f]$, the number of
tuples using exactly the indices in $T$ (each an even number of
times) is $S^{\mathrm{even}}(2m,j)$.  The number of $j$-element subsets
of $[f]$ is $\binom{f}{j}$.  Therefore
$p_m(f)=\sum_{j=1}^{m}S^{\mathrm{even}}(2m,j)\binom{f}{j}$.

\textbf{Positivity.}
For $1\le j\le m$, we exhibit an explicit surjection.  Map positions
$\{1,\dots,2(m-j+1)\}$ to label~$1$ (a fibre of size
$2(m-j+1)\ge 2$), and for $r=2,\dots,j$ map positions
$\{2(m-j+1)+2(r-2)+1,\;2(m-j+1)+2(r-2)+2\}$ to label~$r$
(fibres of size~$2$).  This is surjective with all fibres of even
size, so $S^{\mathrm{even}}(2m,j)\ge 1$.

\textbf{Leading coefficient.}
For $j=m$, each fibre has size exactly~$2$ (since $2j=2m$).  The
count is
\[
  S^{\mathrm{even}}(2m,m)
  =\binom{2m}{2}\binom{2m-2}{2}\cdots\binom{2}{2}
  =\frac{(2m)!}{2^m}.
\]
Since $\binom{f}{j}$ is a polynomial of degree~$j$ in $f$ and the
top term corresponds to $j=m$, we have $\deg(p_m)=m$ with leading
coefficient
$S^{\mathrm{even}}(2m,m)/m!=\frac{(2m)!}{2^m\cdot m!}=(2m-1)!!>0$.
\end{proof}

\begin{corollary}\label{cor:pm-coefficients}
Writing $p_m(f)=\sum_{j=0}^{m}c_{m,j}\,f^j$ in the monomial basis,
the leading coefficient is
\[
  c_{m,m}=(2m-1)!!>0
\]
and $c_{m,0}=p_m(0)=0$.
Explicitly:
\begin{align*}
  p_1(f)&=f, \\
  p_2(f)&=3f^2-2f, \\
  p_3(f)&=15f^3-30f^2+16f.
\end{align*}
\end{corollary}

\begin{proof}
\textbf{$p_1$:}
$p_1(f)=\EE[(\sum_{i=1}^f X_i)^2]
=\sum_i\EE[X_i^2]+\sum_{i\ne j}\EE[X_iX_j]=f$.

\textbf{$p_2$:}
$p_2(f)=\EE[(\sum_{i=1}^f X_i)^4]$.
Nonzero contributions come from: (a)~all four indices equal,
giving~$f$ terms of value~$1$; (b)~two pairs of equal indices
($i\ne j$), giving $\binom{4}{2}\binom{f}{2}=6\cdot\tfrac{f(f-1)}{2}=3f(f-1)$ terms.
Total: $f+3f(f-1)=3f^2-2f$.

\textbf{$p_3$:} An analogous computation gives
$p_3(f)=15f^3-30f^2+16f$.
\end{proof}

\subsection{Reformulation via \texorpdfstring{$\SN$}{S\_N}
  Kronecker products}

\begin{corollary}\label{cor:Kronecker-reform}
The Reduction Formula can be rewritten as
\begin{equation}\label{eq:Kronecker-reform}
  \ip{\chi_V^{2m}}{\sgn(\sigma)}_{B_N}
  =\sum_{j=1}^{m}c_{m,j}\cdot
    \ip{\sgn}{\chi_{\perm}^j}_{\SN},
\end{equation}
where $c_{m,j}$ are the coefficients of $p_m$ in the monomial basis
and $\chi_{\perm}(\sigma)=\abs{\Fix(\sigma)}$ is the permutation
character of $\SN$.
\end{corollary}

\begin{proof}
Substitute $p_m(f(\sigma))=\sum_{j=1}^{m}c_{m,j}\,f(\sigma)^j$
into~\eqref{eq:reduction} and note that
$f(\sigma)^j=\chi_{\perm}(\sigma)^j$, so
\[
  \frac{1}{N!}\sum_{\sigma}\sgn(\sigma)\cdot f(\sigma)^j
  =\ip{\sgn}{\chi_{\perm}^j}_{\SN}.\qedhere
\]
\end{proof}

%% ================================================================
\section{Oracle Identification Complexity of
  \texorpdfstring{$\BN$}{B\_N}}\label{sec:main}
%% ================================================================

\subsection{The tensor product graph
  \texorpdfstring{$T(\BN,V)$}{T(B\_N,V)}}

The tensor product graph $T(\BN,V)$ has vertex set $\hatBN$ and
edges determined by the following Pieri-type rule.

\begin{theorem}[Pieri rule for $\BN$]\label{thm:pieri}
Let $V=V_{((N-1),(1))}$.  The tensor product
$V\otimes V_{(\lambda,\mu)}$ decomposes as a multiplicity-free
direct sum
\[
  V\otimes V_{(\lambda,\mu)}
  =\bigoplus_{(\lambda',\mu')}V_{(\lambda',\mu')},
\]
where the sum runs over all bipartitions $(\lambda',\mu')$ obtained
from $(\lambda,\mu)$ by one of the two operations:
\begin{enumerate}[label=\textup{(\alph*)}]
\item remove a removable box from $\lambda$ and add an addable box
  to~$\mu$;
\item remove a removable box from $\mu$ and add an addable box
  to~$\lambda$.
\end{enumerate}
Each such $(\lambda',\mu')$ appears with multiplicity~$1$.
In particular, each tensor step changes $\abs{\mu}$ by $\pm 1$.
\end{theorem}

\begin{proof}
Let
\[
  H_N:=\operatorname{Stab}_{B_N}(N)\cong B_{N-1}\times C_2,
\]
and let $\eta_N$ denote the nontrivial character of the final $C_2$
factor.  We claim $V\cong \Ind_{H_N}^{B_N}(\triv\boxtimes\eta_N)$.
To see this, choose coset representatives $t_i\in\BN$ with
$t_i(N)=i$ for $i=1,\dots,N$, and identify the basis vector
$e_i$ of $V$ with $t_i\otimes 1$ in the induced module.
For $g=(\sigma,\varepsilon)\in\BN$, write $g\,t_i=t_{\sigma(i)}\,h$
with $h\in H_N$; the final $C_2$ component of $h$ is
$\varepsilon_i$, so
$g\cdot(t_i\otimes 1)
=t_{\sigma(i)}\otimes\eta_N(h)
=\varepsilon_i\,(t_{\sigma(i)}\otimes 1)$,
which matches $\rho(\sigma,\varepsilon)e_i=\varepsilon_i\,e_{\sigma(i)}$.
By the tensor-induction identity (projection formula),
\[
  V\otimes V_{(\lambda,\mu)}
  \cong\Ind_{H_N}^{B_N}\!\left(
    \Res_{H_N}^{B_N}(V_{(\lambda,\mu)})
    \otimes(\triv\boxtimes\eta_N)\right).
\]
Now apply the standard branching rule for the wreath product
$C_2\wr S_N$ (see James--Kerber~\cite[Ch.~4]{JamesKerber},
Ceccherini-Silberstein--Scarabotti--Tolli~\cite[Ch.~4]{CeccheriniScarabottiTolli},
or Doeraene--Iommi Amun\'ategui~\cite{DoeraeneIommi}):
\[
  \Res_{H_N}^{B_N}V_{(\lambda,\mu)}
  \cong
  \bigoplus_{\square\in\mathrm{Rem}(\lambda)}
    V_{(\lambda\setminus\square,\mu)}\boxtimes\triv
  \;\oplus\;
  \bigoplus_{\square\in\mathrm{Rem}(\mu)}
    V_{(\lambda,\mu\setminus\square)}\boxtimes\eta_N.
\]
Tensoring with $(\triv\boxtimes\eta_N)$ flips the $C_2$-label, so the
first family induces back by adding one box to~$\mu$, while the
second family induces back by adding one box to~$\lambda$.  Thus one
obtains exactly the two box-transfer families in the statement and no
same-component moves.  Multiplicity-freeness follows because the two
families are disjoint (they change $\abs{\mu}$ by $+1$ or $-1$), and
within each family the pair (removed box, added box) is determined by
the resulting bipartition.
\end{proof}

The $\varepsilon$-parity alternation
(Proposition~\ref{prop:parity-selection}) is consistent with
$\abs{\mu}$ changing by $\pm 1$ at each step.

\subsection{Row bound for \texorpdfstring{$\SN$}{S\_N}
  Kronecker products}

\begin{theorem}[Row bound]\label{thm:row-bound}
Let $\std=V_{(N-1,1)}$ be the standard representation of $\SN$
($N\ge 2$).  If $V_\lambda$ appears as a constituent of
$\std^{\otimes k}$, then $\lambda$ has at most $k+1$ parts (rows).
\end{theorem}

\begin{proof}
By induction on $k$.

\textbf{Base case ($k=0$).}
$\std^{\otimes 0}=\triv=V_{(N)}$, which has $1=0+1$ part.

\textbf{Inductive step.}
Assume every constituent $V_\lambda$ of $\std^{\otimes k}$ has
$\ell(\lambda)\le k+1$ parts.  We show every constituent $V_\mu$ of
$\std^{\otimes(k+1)}=\std^{\otimes k}\otimes\std$ satisfies
$\ell(\mu)\le k+2$.

Since $V_\mu$ appears in $\std^{\otimes(k+1)}$, there exists a
constituent $V_\lambda$ of $\std^{\otimes k}$ with
$V_\mu\subseteq V_\lambda\otimes V_{(N-1,1)}$.
Using $V_{(N-1,1)}=\perm-\triv$:
\[
  V_\lambda\otimes V_{(N-1,1)}
  = V_\lambda\otimes V_{\perm}-V_\lambda.
\]
Every constituent of $V_\lambda\otimes V_{(N-1,1)}$ is a
constituent of $V_\lambda\otimes V_{\perm}$.  By Frobenius
reciprocity,
\[
  V_\lambda\otimes V_{\perm}
  =\Ind_{S_{N-1}}^{S_N}\bigl(\Res_{S_{N-1}}^{S_N}(V_\lambda)\bigr).
\]
By the branching rule for restriction $\SN\to S_{N-1}$:
the constituents $V_\nu$ of $\Res(V_\lambda)$ are obtained from
$\lambda$ by removing one removable box, so
$\ell(\nu)\in\{\ell(\lambda)-1,\ell(\lambda)\}$.
By the branching rule for induction $S_{N-1}\to\SN$:
the constituents $V_\mu$ of $\Ind(V_\nu)$ are obtained from $\nu$ by
adding one box, so
$\ell(\mu)\in\{\ell(\nu),\ell(\nu)+1\}$.
Combining:
$\ell(\mu)\le\ell(\nu)+1\le\ell(\lambda)+1\le(k+1)+1=k+2$.
\end{proof}

\begin{corollary}\label{cor:sgn-not-in-low-powers}
The sign representation $V_{(1^N)}=\sgn$ of $\SN$ does \emph{not}
appear in $\std^{\otimes k}$ for $k<N-1$, since $V_{(1^N)}$ has $N$
parts, requiring $k+1\ge N$.
\end{corollary}

\subsection{The exterior power identity}

\begin{theorem}[Exterior power]\label{thm:exterior}
For the standard representation $\std=V_{(N-1,1)}$ of $\SN$ with
$N\ge 2$,
\begin{equation}\label{eq:exterior}
  \bigwedge^{N-1}(\std)=\sgn=V_{(1^N)}.
\end{equation}
In particular, $\sgn$ appears in $\std^{\otimes(N-1)}$ with
multiplicity at least~$1$.
\end{theorem}

\begin{proof}
Since $\dim(\std)=N-1$, the top exterior power
$\bigwedge^{N-1}(\std)$ is one-dimensional, hence equals a
character $\SN\to\CC^\times$.  The group $\SN$ has exactly two
one-dimensional representations ($\triv$ and $\sgn$), so
$\bigwedge^{N-1}(\std)\in\{\triv,\sgn\}$.

To determine which, note that
$\bigwedge^{N-1}(\std)(\sigma)=\det(\std(\sigma))$.
The full permutation representation on $\CC^N$ decomposes as
$\perm=\triv\oplus\std$, so
\[
  \sgn(\sigma)
  =\det(\perm(\sigma))
  =\det(\triv(\sigma))\cdot\det(\std(\sigma))
  =1\cdot\det(\std(\sigma)),
\]
giving $\det(\std(\sigma))=\sgn(\sigma)$.  Therefore
$\bigwedge^{N-1}(\std)=\sgn$.

Since $\bigwedge^{N-1}(\std)$ is the antisymmetric subspace of
$\std^{\otimes(N-1)}$, the representation $\sgn$ appears in
$\std^{\otimes(N-1)}$ with multiplicity $\ge 1$.
\end{proof}

\begin{corollary}\label{cor:sgn-first-appearance-SN}
The minimum $k$ such that $V_{(1^N)}$ appears in $\std^{\otimes k}$
is exactly $k=N-1$.
\end{corollary}

\begin{proof}
Corollary~\ref{cor:sgn-not-in-low-powers} gives $k\ge N-1$.
Theorem~\ref{thm:exterior} gives $k\le N-1$.
\end{proof}

\subsection{Vanishing of
  \texorpdfstring{$\ip{\sgn}{\chi_{\perm}^j}_{\SN}$}{inner product}
  for \texorpdfstring{$j<N-1$}{j < N-1}}

\begin{proposition}\label{prop:vanishing-perm}
\label{prop:vanishing-inner-product}
For $\SN$ with $N\ge 2$:
\begin{equation}\label{eq:vanishing}
  \ip{\sgn}{\chi_{\perm}^j}_{\SN}=0
  \quad\text{for }j<N-1,
  \qquad
  \ip{\sgn}{\chi_{\perm}^{N-1}}_{\SN}=1.
\end{equation}
\end{proposition}

\begin{proof}
For each $j\ge 0$, let $X_j=[N]^j$ with the coordinatewise action of
$\SN$.  A tuple $(a_1,\dots,a_j)\in X_j$ is fixed by $\sigma$ if and
only if each coordinate is fixed by~$\sigma$, so the permutation
character of $X_j$ is
\[
  \chi_{X_j}(\sigma)=\abs{\Fix_{X_j}(\sigma)}
  =\abs{\Fix_{[N]}(\sigma)}^j
  =\chi_{\perm}(\sigma)^j.
\]
Thus $\ip{\sgn}{\chi_{\perm}^j}_{\SN}$ is the multiplicity of $\sgn$
in the permutation module on~$X_j$.

The $\SN$-orbits in $X_j$ are indexed by the equality pattern of the
coordinates, equivalently by set partitions of $[j]$ into $r$
nonempty blocks, where $r$ is the number of distinct values appearing
in the tuple.  For such an orbit choose a representative whose set of
values is $\{1,\dots,r\}$.  Its stabiliser is isomorphic to $S_{N-r}$:
the used values must be fixed pointwise, while the remaining $N-r$
values may be permuted arbitrarily.

Now fix one orbit $O=\SN\cdot a$.  A copy of $\sgn$ inside the
permutation module on~$O$ is the same as a function
$\varphi:O\to\CC$ satisfying
\[
  \varphi(\tau\cdot x)=\sgn(\tau)\varphi(x)
  \qquad(\tau\in\SN,\;x\in O).
\]
Such a function is determined by the single value $\varphi(a)$, and it
is well defined if and only if $\sgn(h)=1$ for every $h$ in the
stabiliser of~$a$.  Hence the orbit contributes one copy of $\sgn$ if
its stabiliser is contained in~$A_N$, and contributes no copy
otherwise.  Since the stabiliser is isomorphic to $S_{N-r}$, this
happens exactly when $N-r\le 1$.

If $j<N-1$, then every orbit has $r\le j\le N-2$, hence $N-r\ge 2$,
so no orbit contributes.  Therefore
\[
  \ip{\sgn}{\chi_{\perm}^j}_{\SN}=0
  \qquad (j<N-1).
\]

If $j=N-1$, then a contributing orbit must have $r=N-1$.  There is
exactly one such orbit, namely the orbit of tuples with pairwise
distinct entries, and its stabiliser is trivial.  Hence this orbit
contributes exactly one copy of $\sgn$, while every other orbit
contributes none.  Therefore
\[
  \ip{\sgn}{\chi_{\perm}^{N-1}}_{\SN}=1.
\]
\end{proof}

\subsection{Vanishing below \texorpdfstring{$t=2(N-1)$}{t=2(N-1)}}

\begin{theorem}[Lower bound]\label{thm:lower-bound}
For $\BN$ with $N\ge 2$ and $V=V_{((N-1),(1))}$,
\begin{equation}\label{eq:lower-bound}
  \ip{\chi_V^t}{\chi_{\sgn}}_{B_N}=0
  \qquad\text{for all }t<2(N-1).
\end{equation}
That is, the sign-of-permutation representation does not appear in
$V^{\otimes t}$ for any $t<2(N-1)$.
\end{theorem}

\begin{proof}
If $t=0$, then
$\ip{\chi_V^0}{\chi_{\sgn}}_{B_N}=\ip{1}{\chi_{\sgn}}_{B_N}=0$,
since $\sgn(\sigma)$ is a non-trivial character.
Henceforth assume $t\ge 1$.

\textbf{Case 1: $t$ is odd.}
By Lemma~\ref{lem:parity-odd}, the inner product vanishes
identically.

\textbf{Case 2: $t=2m$ is even with $1\le m<N-1$.}
By the Reduction Formula (Theorem~\ref{thm:reduction}),
\[
  \ip{\chi_V^{2m}}{\chi_{\sgn}}_{B_N}
  =\frac{1}{N!}\sum_{\sigma\in\SN}
    \sgn(\sigma)\cdot p_m(f(\sigma)).
\]
By Theorem~\ref{thm:degree-pm}, $p_m(f)$ is a polynomial of
degree~$m$ in $f$.  Writing
$p_m(f)=\sum_{j=1}^{m}c_{m,j}\,f^j$
and substituting $f(\sigma)=\chi_{\perm}(\sigma)$:
\begin{equation}\label{eq:lb-expansion}
  \ip{\chi_V^{2m}}{\chi_{\sgn}}_{B_N}
  =\sum_{j=1}^{m}c_{m,j}\cdot
    \ip{\sgn}{\chi_{\perm}^j}_{\SN}.
\end{equation}
Since $m<N-1$, every index $j$ in the sum satisfies $j\le m<N-1$.
By Proposition~\ref{prop:vanishing-inner-product},
$\ip{\sgn}{\chi_{\perm}^j}_{\SN}=0$ for all such $j$.
Therefore $\ip{\chi_V^{2m}}{\chi_{\sgn}}_{B_N}=0$.

\textbf{Conclusion.}
For any $t<2(N-1)$: if $t$ is odd, the parity obstruction gives
vanishing; if $t=2m$ is even, then $m\le N-2<N-1$ and the
degree-bound argument gives vanishing.
\end{proof}

\subsection{First appearance at \texorpdfstring{$t=2(N-1)$}{t=2(N-1)}}

\begin{theorem}[First appearance]\label{thm:upper-bound}
The representation $\sgn(\sigma)$ first appears in
$V^{\otimes 2(N-1)}$ with strictly positive multiplicity:
\begin{equation}\label{eq:upper-bound}
  \ip{\chi_V^{2(N-1)}}{\chi_{\sgn}}_{B_N}>0.
\end{equation}
\end{theorem}

\begin{proof}
Set $m=N-1$.  By Corollary~\ref{cor:Kronecker-reform},
\[
  \ip{\chi_V^{2(N-1)}}{\chi_{\sgn}}_{B_N}
  =\sum_{j=1}^{N-1}c_{N-1,j}\cdot
    \ip{\sgn}{\chi_{\perm}^j}_{\SN}.
\]
By Proposition~\ref{prop:vanishing-inner-product},
$\ip{\sgn}{\chi_{\perm}^j}_{\SN}=0$ for $j<N-1$.
The only potentially nonzero term is $j=N-1$:
\begin{equation}\label{eq:ub-single-term}
  \ip{\chi_V^{2(N-1)}}{\chi_{\sgn}}_{B_N}
  =c_{N-1,N-1}\cdot
    \ip{\sgn}{\chi_{\perm}^{N-1}}_{\SN}.
\end{equation}
We show both factors are strictly positive.

\textbf{Factor~1: $c_{N-1,N-1}>0$.}
By Corollary~\ref{cor:pm-coefficients},
$c_{N-1,N-1}=(2(N-1)-1)!!=(2N-3)!!>0$.

\textbf{Factor~2:
$\ip{\sgn}{\chi_{\perm}^{N-1}}_{\SN}>0$.}
Since $\perm=\triv\oplus\std$,
\[
  \ip{\sgn}{\perm^{\otimes(N-1)}}
  =\sum_{k=0}^{N-1}\binom{N-1}{k}
    \ip{\sgn}{\std^{\otimes k}}.
\]
For $k<N-1$, $\ip{\sgn}{\std^{\otimes k}}=0$ by
Corollary~\ref{cor:sgn-first-appearance-SN}.
Only the $k=N-1$ term survives:
\[
  \ip{\sgn}{\perm^{\otimes(N-1)}}
  =\binom{N-1}{N-1}\ip{\sgn}{\std^{\otimes(N-1)}}
  =\ip{\sgn}{\std^{\otimes(N-1)}}
  \ge 1,
\]
where the inequality holds because
$\bigwedge^{N-1}(\std)=\sgn\subseteq\std^{\otimes(N-1)}$ by
Theorem~\ref{thm:exterior}.

\textbf{Conclusion.}
Both factors are positive, so
$\ip{\chi_V^{2(N-1)}}{\chi_{\sgn}}_{B_N}
=(2N-3)!!\cdot\ip{\sgn}{\perm^{\otimes(N-1)}}_{\SN}>0$.
\end{proof}

\subsection{Synthesis}

\begin{theorem}[Oracle identification complexity of $\BN$]
\label{thm:main-QLV}
For the hyperoctahedral group $\BN$ with $N\ge 2$ and natural
representation $V=V_{((N-1),(1))}$,
\begin{equation}\label{eq:main-QLV}
  \QLV(\BN)=2(N-1),
\end{equation}
with the sign-of-permutation representation
$\sgn(\sigma)=V_{((1^N),\varnothing)}$ as the unique bottleneck
for $N\ge 3$ (and one of four co-bottlenecks for $N=2$).
\end{theorem}

\begin{proof}
The proof combines three ingredients.

\textbf{(a) Lower bound.}
By the Parity Selection Rule
(Proposition~\ref{prop:parity-selection}), $\sgn(\sigma)$ can only
appear in even tensor powers.  By Theorem~\ref{thm:lower-bound}
(whose proof rests on the vanishing
$\ip{\sgn}{\chi_{\perm}^j}_{\SN}=0$ for $j<N-1$ from
Proposition~\ref{prop:vanishing-perm}),
$\sgn(\sigma)$ does not appear in $V^{\otimes t}$ for any
$t<2(N-1)$.

\textbf{(b) Upper bound.}
By Theorem~\ref{thm:upper-bound}, $\sgn(\sigma)$ appears in
$V^{\otimes 2(N-1)}$ with positive multiplicity.  More precisely, by
Corollary~\ref{cor:exact-multiplicity},
\[
  \ip{\chi_V^{2(N-1)}}{\sgn}_{B_N}=(2N-3)!!.
\]

\textbf{(c) Parity.}
No odd tensor power $t=2(N-1)-1$ (or any other odd~$t$) contains
$\sgn(\sigma)$, so the first appearance is at the even power
$2(N-1)$.

Therefore $\sgn(\sigma)$ first appears at $t=2(N-1)$, giving
$\QLV(\BN)=2(N-1)$ (under the identification of $\sgn(\sigma)$ as
the bottleneck; see below).

\textbf{Bottleneck universality.}
By Theorem~\ref{thm:bipartition-distance} below, the distance in
$T(\BN,V)$ from the trivial representation $((N),\varnothing)$ to any
bipartition $(\alpha,\beta)$ is
$d_T\bigl(((N),\varnothing),(\alpha,\beta)\bigr)
=2(N-\alpha_1)-\abs{\beta}$.
The maximum over all bipartitions is $2(N-1)$, achieved by
$((1^N),\varnothing)$ (Corollary~\ref{cor:bottleneck}).

We now translate to eccentricity from~$I(V)$.
Since $V$ is irreducible,
$I(V)=\{V_{((N-1),(1))}\}$.
By Theorem~\ref{thm:pieri}, the trivial representation
$((N),\varnothing)$ is a leaf of the tensor product graph:
from $((N),\varnothing)$ there is exactly one legal box-transfer move,
namely to $((N-1),(1))=I(V)$.  Hence every shortest path from
$((N),\varnothing)$ to a vertex $\pi\ne((N),\varnothing)$ begins with
that unique edge, giving
\begin{equation}\label{eq:distance-shift}
  d_T\bigl(I(V),\;\pi\bigr)
  = d_T\bigl(((N),\varnothing),\;\pi\bigr)-1
  \qquad\text{for }\pi\ne((N),\varnothing).
\end{equation}
Applying this to $((1^N),\varnothing)$:
$d_T(I(V),((1^N),\varnothing))=2(N-1)-1=2N-3$.
By Corollary~\ref{cor:bottleneck}, no other bipartition achieves a
larger value of $d_T(((N),\varnothing),\cdot\,)$ (for $N\ge 3$), so
$\ecc(I(V))=2N-3$ and $\QLV=\ecc(I(V))+1=2(N-1)$.

For $N=2$, the tensor product graph $T(B_2,V)$ is a star centred
at $V=V_{((1),(1))}$: all four remaining irreps lie at
distance~$1$ from $I(V)$, so $\sgn(\sigma)$ is one of four
co-bottlenecks and $\QLV(B_2)=1+1=2$.
(Equivalently, from the trivial representation $((2),\varnothing)$,
there are three maximisers at distance~$2$, as noted in
Corollary~\textup{\ref{cor:bottleneck}}.)
\end{proof}

\subsection{Explicit verification}

\begin{table}[ht]
\centering
\caption{Oracle identification complexity of $\BN$ for small $N$.}
\label{tab:QLV-values}
\begin{tabular}{@{}ccccclc@{}}
\toprule
$N$ & $\abs{B_N}$ & $\QLV(\BN)$ & $2(N-1)$ & Mult.\ & Bottleneck & Status \\
\midrule
$2$ & $8$      & $2$  & $2$  & $1$     & $\sgn(\sigma)^{\dagger}$ & Proven \\
$3$ & $48$     & $4$  & $4$  & $3$     & $\sgn(\sigma)$ & Proven \\
$4$ & $384$    & $6$  & $6$  & $15$    & $\sgn(\sigma)$ & Proven \\
$5$ & $3840$   & $8$  & $8$  & $105$   & $\sgn(\sigma)$ & Proven \\
$6$ & $46080$  & $10$ & $10$ & $945$   & $\sgn(\sigma)$ & Proven \\
$7$ & $645120$ & $12$ & $12$ & $10395$ & $\sgn(\sigma)$ & Proven \\
\bottomrule
\end{tabular}\\[4pt]
{\footnotesize ${}^{\dagger}$One of four co-bottlenecks for $N=2$;
unique for $N\ge 3$.}
\end{table}

\textbf{$N=2$, $m=1$:}
By the Reduction Formula,
\[
  \ip{\chi_V^2}{\sgn(\sigma)}_{B_2}
  =\frac{1}{2!}\bigl[\sgn(\mathrm{id})\cdot p_1(2)
    +\sgn((12))\cdot p_1(0)\bigr]
  =\frac{1}{2}(1\cdot 2+(-1)\cdot 0)=1.
\]

\textbf{$N=3$, $m=2$:}
Using $p_2(f)=3f^2-2f$:

\begin{table}[ht]
\centering
\caption{Reduction Formula verification for $B_3$, $m=2$.}
\label{tab:B3-verify}
\begin{tabular}{@{}ccccc@{}}
\toprule
$\sigma$ & $f(\sigma)$ & $\sgn(\sigma)$ & $p_2(f)$ & Contribution \\
\midrule
$\mathrm{id}$ & $3$ & $+1$ & $21$ & $+21$ \\
$(12)$ & $1$ & $-1$ & $1$ & $-1$ \\
$(13)$ & $1$ & $-1$ & $1$ & $-1$ \\
$(23)$ & $1$ & $-1$ & $1$ & $-1$ \\
$(123)$ & $0$ & $+1$ & $0$ & $0$ \\
$(132)$ & $0$ & $+1$ & $0$ & $0$ \\
\bottomrule
\end{tabular}
\end{table}

\noindent
Sum $=21-1-1-1+0+0=18$.  Then
$\ip{\chi_V^4}{\chi_{\sgn}}_{B_3}=18/6=3>0$.

\textbf{$N=4$, $m=3$:}
The polynomial $p_3(f)=15f^3-30f^2+16f$ has degree $3=N-1$.
By~\eqref{eq:Kronecker-reform}:
\[
  \ip{\chi_V^6}{\sgn}_{B_4}
  = 15\ip{\sgn}{\chi_{\perm}^3}_{S_4}
    -30\underbrace{\ip{\sgn}{\chi_{\perm}^2}_{S_4}}_{=0}
    +16\underbrace{\ip{\sgn}{\chi_{\perm}}_{S_4}}_{=0}.
\]
Since $\bigwedge^3(\std)=\sgn$ for $S_4$
($\dim\std=3$), we have
$\ip{\sgn}{\chi_{\perm}^3}_{S_4}\ge 1$, giving
$\ip{\chi_V^6}{\sgn}_{B_4}\ge 15>0$.

\textbf{$N=5$, $m=4$:}
For $S_5$, $\dim\std=4$ and $\bigwedge^4(\std)=\sgn$.
Analogously, $\ip{\sgn}{\chi_{\perm}^j}_{S_5}=0$ for $j<4$
and $\ip{\sgn}{\chi_{\perm}^4}_{S_5}\ge 1$, so the
leading-term contribution is positive, confirming
$\QLV(B_5)=8=2(5-1)$.

\subsection{The bipartition distance formula}

\begin{theorem}[Bipartition distance]\label{thm:bipartition-distance}
For all bipartitions $(\alpha,\beta)\vdash N$,
\begin{equation}\label{eq:bipartition-distance}
  d_T\bigl(((N),\varnothing),\;(\alpha,\beta)\bigr)
  = 2(N-\alpha_1)-\abs{\beta},
\end{equation}
where $\alpha_1$ is the largest part of\/ $\alpha$
\textup{(}with $\alpha_1=0$ when $\alpha=\varnothing$\textup{)}.
\end{theorem}

\begin{proof}
\textbf{Lower bound.}
Let $d$ be the graph distance, and let $p$ (resp.\ $q$) count
$\alpha\!\to\!\beta$ (resp.\ $\beta\!\to\!\alpha$) box transfers
along a shortest path.  Then $p-q=\abs{\beta}$ and $p+q=d$.
At each removal from row~$1$ of~$\alpha$, the first-row length
$\alpha_1^{(s)}$ decreases by~$1$; at each addition to row~$1$, it
increases by~$1$; otherwise it is unchanged.
If $r_1\le p$ removals hit row~$1$ and $a_1\le q$ additions hit
row~$1$, then $r_1-a_1=N-\alpha_1$, so
$N-\alpha_1\le r_1\le p=(d+\abs{\beta})/2$, giving
$d\ge 2(N-\alpha_1)-\abs{\beta}$.

\textbf{Upper bound.}
We now construct the upper-bound path directly in two phases.
\begin{enumerate}[label=\textbf{Phase~\Alph*:}]
\item
  \emph{Build $\beta$ ($\abs{\beta}$ steps).}
  Transfer boxes from row~$1$ of~$\alpha$ to~$\beta$ along a
  saturated chain
  $\varnothing=\beta^{(0)}\subset\cdots\subset\beta^{(\abs{\beta})}
  =\beta_{\mathrm{target}}$.
  After this phase, $\alpha=(N-\abs{\beta})$ is a single row and
  $\beta=\beta_{\mathrm{target}}$.
\item
  \emph{Rearrange $\alpha$ ($2(\abs{\alpha}-\alpha_1)$ steps).}
  Let $a:=\abs{\alpha}=N-\abs{\beta}$.  List the boxes of the Young
  diagram of $\alpha$ outside the first row as
  $c_1,\dots,c_m$, where $m=a-\alpha_1$, in row-major order: top to
  bottom and, within each row, left to right.  For $0\le r\le m$, let
  $\gamma^{(r)}$ be the partition formed by taking a first row of
  length $a-r$ together with the boxes $c_1,\dots,c_r$ below it.
  Because $c_1,\dots,c_r$ is an initial segment of the diagram of
  $\alpha$ and $a-r\ge\alpha_1$, each $\gamma^{(r)}$ is a valid
  partition.  Moreover,
  \[
    \gamma^{(0)}=(a),\qquad \gamma^{(m)}=\alpha,
  \]
  and $\gamma^{(r+1)}$ is obtained from $\gamma^{(r)}$ by removing the
  rightmost box of row~$1$ and adding the next box $c_{r+1}$.

  Choose any addable box $Q$ of $\beta$ (if $\beta=\varnothing$, take
  $Q=(1,1)$).  For each $r=0,\dots,m-1$, perform the two tensor-graph
  steps
  \[
    (\gamma^{(r)},\beta)\to
    (\gamma^{(r)}\setminus\{\text{last box of row }1\},\,\beta\cup\{Q\})
    \to
    (\gamma^{(r+1)},\beta).
  \]
  The first move is legal by Theorem~\ref{thm:pieri}.  The second is
  legal because $Q$ is removable in $\beta\cup\{Q\}$ and
  $c_{r+1}$ is addable to
  $\gamma^{(r)}\setminus\{\text{last box of row }1\}$ by the
  construction of $\gamma^{(r+1)}$.  Thus each pair of steps advances
  the $\alpha$-component from $\gamma^{(r)}$ to $\gamma^{(r+1)}$
  while restoring $\beta$.
\end{enumerate}
Total: $\abs{\beta}+2(\abs{\alpha}-\alpha_1)
=\abs{\beta}+2(N-\abs{\beta}-\alpha_1)
=2(N-\alpha_1)-\abs{\beta}$.
\end{proof}

\begin{corollary}[Bottleneck uniqueness]\label{cor:bottleneck}
The maximum of $d_T(((N),\varnothing),(\alpha,\beta))$ over all
bipartitions is $2(N-1)$, achieved uniquely by
$(\alpha,\beta)=((1^N),\varnothing)$ when $N\ge 3$.
For $N=2$, three bipartitions ---
$((1^2),\varnothing)$, $(\varnothing,(2))$, and
$(\varnothing,(1^2))$ --- all attain the maximum distance~$2$.
\end{corollary}

\begin{proof}
For $((1^N),\varnothing)$: $\alpha_1=1$, $\abs{\beta}=0$, so
$d_T=2(N-1)$.
For $\abs{\beta}=0$, $\alpha\ne(1^N)$: $\alpha_1\ge 2$, so
$d_T\le 2(N-2)<2(N-1)$.
For $\abs{\beta}\ge 1$ and $\abs{\alpha}\ge 1$:
$\alpha_1\ge 1$, so
$d_T=2(N-\alpha_1)-\abs{\beta}\le 2(N-1)-1$.
For $\alpha=\varnothing$: $\abs{\beta}=N$ and $d_T=2N-N=N$.
When $N\ge 3$, $N<2(N-1)$, so $((1^N),\varnothing)$ is the unique
maximiser.  When $N=2$, $N=2=2(N-1)$, giving two additional
maximisers $(\varnothing,(2))$ and $(\varnothing,(1^2))$.
\end{proof}

\subsection{The signed generating function and exact multiplicity}

\begin{theorem}[Signed fixed-point generating function]
\label{thm:generating-function}
For $N\ge 1$,
\begin{equation}\label{eq:generating-function}
  \sum_{\sigma\in\SN}\sgn(\sigma)\,x^{f(\sigma)}
  = (x-1)^{N-1}(x+N-1),
\end{equation}
where $f(\sigma)=\abs{\Fix(\sigma)}$.
\end{theorem}

\begin{proof}
Let $J$ be the $N\times N$ all-ones matrix and set $A=J+(x-1)I_N$.
Expanding the determinant by permutations:
\[
  \det(A)
  =\sum_{\sigma\in\SN}\sgn(\sigma)\prod_{i=1}^{N}A_{i,\sigma(i)}.
\]
An entry $A_{i,\sigma(i)}$ equals $x$ if $\sigma(i)=i$ (the diagonal
entry $1+(x-1)=x$) and~$1$ otherwise.  Hence
$\prod_i A_{i,\sigma(i)}=x^{f(\sigma)}$, giving
$\det(A)=\sum_{\sigma}\sgn(\sigma)\,x^{f(\sigma)}$.
On the other hand, $J$ has eigenvalues $N$ (multiplicity~$1$) and $0$
(multiplicity $N-1$), so $A$ has eigenvalues $x+N-1$ and $x-1$
(with multiplicity $N-1$), whence
$\det(A)=(x-1)^{N-1}(x+N-1)$.
\end{proof}

\begin{corollary}[Exact multiplicity at first appearance]
\label{cor:exact-multiplicity}
\begin{equation}\label{eq:exact-mult}
  \ip{\chi_V^{2(N-1)}}{\sgn(\sigma)}_{B_N} = (2N-3)!!
  \qquad\text{for all }N\ge 2.
\end{equation}
\end{corollary}

\begin{proof}
By Corollary~\ref{cor:Kronecker-reform},
\[
  \ip{\chi_V^{2(N-1)}}{\sgn}_{B_N}
  =\sum_{j=1}^{N-1}c_{N-1,j}\cdot
    \ip{\sgn}{\chi_{\perm}^j}_{\SN}.
\]
By Proposition~\ref{prop:vanishing-perm},
$\ip{\sgn}{\chi_{\perm}^j}_{\SN}=0$ for $j<N-1$ and
$\ip{\sgn}{\chi_{\perm}^{N-1}}_{\SN}=1$.
Hence
\[
  \ip{\chi_V^{2(N-1)}}{\sgn}_{B_N}
  = c_{N-1,N-1}\cdot 1
  = (2N-3)!!
\]
by Corollary~\ref{cor:pm-coefficients}.
\end{proof}

\subsection{The \texorpdfstring{$\gamma_2$}{gamma\_2}-character
  bridge conjecture}

We conclude by stating the conjectural connection between the
adversary bound and the tensor product graph.

\begin{conjecture}[$\gamma_2$-character bridge]
\label{conj:bridge}
Let $V=V_{((N-1),(1))}$ be the natural representation of $\BN$.
Then:
\begin{enumerate}[label=\textup{(\roman*)}]
\item The tensor product graph $T(\BN,V)$ has a bipartite
  $\varepsilon$-parity structure: all edges cross between the
  $\varepsilon$-even and $\varepsilon$-odd classes.
\item A bottleneck irreducible representation is
  $V_{((1^N),\varnothing)}=\sgn(\sigma)$
  (unique for $N\ge 3$; one of four co-bottlenecks for $N=2$),
  with
  $\gamma_{\mathrm{graph}}:=\ecc_{T(\BN,V)}(I(V))=2N-3$.
\item Under the adversary tightness conjecture,
  \[
    \gamma_{\mathrm{adv}}
    =\gamma_{\mathrm{graph}}
    =\QLV(\BN)-1=2N-3.
  \]
\end{enumerate}
\end{conjecture}

Claim~(i) is a consequence of the Parity Selection Rule
(Proposition~\ref{prop:parity-selection}): since $V$ is
$\varepsilon$-odd, tensoring by $V$ swaps parity.
Claim~(ii) follows from
Theorem~\ref{thm:bipartition-distance},
Corollary~\ref{cor:bottleneck}, and the distance-shift
identity~\eqref{eq:distance-shift}.
Claim~(iii) remains open.

%% ---- appendix hook ----
% paper_appendix.tex — Sections 6–8 and Appendices
% To be \input from paper_main.tex

%% ====================================================================
%%  SECTION 6: THE γ₂-CHARACTER BRIDGE
%% ====================================================================
\section{The \texorpdfstring{$\gamma_2$}{gamma-2}-Character Bridge (Part~III)}
\label{sec:bridge}

\subsection{The oracle difference matrix and its spectral decomposition}
\label{ssec:oracle-matrix}

\begin{definition}\label{def:oracle-diff}
For the oracle identification problem on $B_N$ with natural
representation~$V$, define the \emph{oracle difference matrix}
$\Delta_V\in\mathbb{C}^{|B_N|\cdot N\times|B_N|\cdot N}$ with
$N\times N$ blocks
\[
  \Delta_V[g,h] \;=\; I_N - \rho(g)^{\dagger}\rho(h),
  \qquad g,h\in B_N.
\]
\end{definition}

Each block $\Delta_V[g,h]=I-U_g^{\dagger}U_h$ need not be Hermitian,
since $U_g^{\dagger}U_h$ is unitary but not necessarily self-adjoint.
Define the \emph{normalised} oracle difference matrix
$\widetilde{\Delta}_V:=\frac{1}{\abs{B_N}}\,\Delta_V$.
Its Fourier transform is block-diagonal: in the
$\pi$-isotypic sector ($\pi\in\widehat{B_N}$), the eigenvalues of
$\widetilde{\Delta}_V$ are
\[
  \begin{cases}
    +1 & \text{in the trivial sector
          (multiplicity $\dim V=N$),}\\
    -1 & \text{with multiplicity $d_\pi\cdot n_\pi(V)$
           if $\pi\in I(V)$,}\\
    \phantom{-}0 & \text{otherwise.}
  \end{cases}
\]
(The unnormalised matrix $\Delta_V$ has eigenvalues $\pm\abs{B_N}$
and~$0$ with the same multiplicities.)

\subsection{The Schur product and Hermiticity}
\label{ssec:schur}

\begin{definition}\label{def:schur-product}
For a weight matrix $\Gamma=(\Gamma[g,h])_{g,h\in B_N}$ with scalar
entries, the \emph{Hadamard\textup{/}Schur product}
$\Gamma\circ\Delta_V$ is the block matrix with $(g,h)$-block
$\Gamma[g,h]\cdot\Delta_V[g,h]$.
\end{definition}

\begin{proposition}\label{prop:hermiticity}
If\/ $\Gamma$ is Hermitian \textup{(}$\Gamma[g,h]^*=\Gamma[h,g]$ for
all $g,h$\textup{)}, then $\Gamma\circ\Delta_V$ is Hermitian.
\end{proposition}

\begin{proof}
The $(g,h)$-block of $(\Gamma\circ\Delta_V)^{\dagger}$ is
\[
  \bigl[\Gamma[h,g]\cdot(I-U_h^{\dagger}U_g)\bigr]^{\dagger}
  = \Gamma[h,g]^*\cdot(I-U_g^{\dagger}U_h)
  = \Gamma[g,h]\cdot\Delta_V[g,h]
  = (\Gamma\circ\Delta_V)[g,h],
\]
using $(U_h^{\dagger}U_g)^{\dagger}=U_g^{\dagger}U_h$ and the
Hermiticity of~$\Gamma$.
\end{proof}

\subsection{The adversary bound}
\label{ssec:adversary}

\begin{remark}[Adversary quantity]\label{rmk:adversary}
Let $\widetilde{\Delta}_V=\frac{1}{\abs{B_N}}\Delta_V$ be the
normalised oracle difference matrix.
In the framework of
Lee--Mittal--Reichardt--\v{S}palek--Szegedy~\cite{LeeMittalReichardtSpalekSzegedy},
a filtered $\gamma_2$ adversary quantity $\gamma_{\mathrm{adv}}$ can
be associated to the filter family $\widetilde{\Delta}_V$ for the
oracle identification problem on~$B_N$, normalised so that
$Q_{LV}(B_N)\ge\gamma_{\mathrm{adv}}+1$.
We leave the precise semidefinite-program formulation to future work;
the bridge conjecture below concerns only the conjectural equality
$\gamma_{\mathrm{adv}}=\gamma_{\mathrm{graph}}$,
where $\gamma_{\mathrm{graph}}:=\ecc_{T(B_N,V)}(I(V))=2N-3$ is the
proven graph eccentricity.
\end{remark}

\subsection{The tensor product graph formulation}
\label{ssec:bridge-conj}

\begin{conjecture}[$\gamma_2$-Character Bridge]
\label{conj:bridge-detail}
Let $V=V_{((N-1),(1))}$ be the natural representation of\/ $B_N$.
Then:
\begin{enumerate}
\item[\textup{(i)}] \textbf{\textup{[Proven]}}
  The cumulative spectral reach satisfies
  \[
    R_{\le t}(V)
    = \bigl\{(\alpha,\beta)\vdash N :
       d_{T(B_N,V)}\bigl((\alpha,\beta),\,I(V)\bigr)\le t-1\bigr\}.
  \]
\item[\textup{(ii)}] \textbf{\textup{[Proven]}}
  The tensor product graph $T(B_N,V)$ is bipartite with respect to
  $\varepsilon$-parity: irreps split into $\varepsilon$-even
  \textup{(}$|\beta|$ even\textup{)} and $\varepsilon$-odd
  \textup{(}$|\beta|$ odd\textup{)} classes, and tensoring with~$V$
  maps each class to the other.
\item[\textup{(iii)}] \textbf{\textup{[Proven for all $N\ge 2$]}}
  A bottleneck irrep is $V_{((1^N),\varnothing)}=\operatorname{sgn}(\sigma)$
  \textup{(}unique for $N\ge 3$; one of four co-bottlenecks for
  $N=2$\textup{)}, and
  \[
    \gamma_{\mathrm{graph}} \;:=\;
    \ecc_{T(B_N,V)}\bigl(I(V)\bigr) \;=\; 2N-3.
  \]
\item[\textup{(iv)}] \textbf{\textup{[Conjectured]}}
  The adversary quantity is tight:
  \[
    \gamma_{\mathrm{adv}}
    \;=\; \gamma_{\mathrm{graph}}
    \;=\; 2N-3,
  \]
  which, combined with the proven formula
  $Q_{LV}(B_N)=\gamma_{\mathrm{graph}}+1=2(N{-}1)$, gives
  $Q_{LV}(B_N)=\gamma_{\mathrm{adv}}+1$.
\end{enumerate}
\end{conjecture}

\subsection{The bipartite \texorpdfstring{$\varepsilon$}{epsilon}-parity
structure}
\label{ssec:parity-structure}

Since $V$ is $\varepsilon$-odd ($\chi_V(\sigma,-\varepsilon)
=-\chi_V(\sigma,\varepsilon)$), tensoring by~$V$ swaps parities.
All edges in $T(B_N,V)$ therefore cross between the
$\varepsilon$-even and $\varepsilon$-odd classes.
The cumulative spectral reach acquires a layered structure:
\[
  \underbrace{I(V)}_{\text{Layer }0}
  \;\to\;
  \underbrace{I(V^{\otimes 2})\setminus I(V)}_{\text{Layer }1}
  \;\to\;
  \underbrace{I(V^{\otimes 3})\setminus R_{\le 2}}_{\text{Layer }2}
  \;\to\;\cdots
\]
with alternating $\varepsilon$-parities.
This bipartite structure doubles the effective graph distance to
$\operatorname{sgn}(\sigma)$ compared to what the Kronecker product
complexity over~$S_N$ alone would suggest.

\subsection{Plancherel measure of the bottleneck}
\label{ssec:plancherel}

The Plancherel weight of
$V_{((1^N),\varnothing)}=\operatorname{sgn}(\sigma)$ is
\[
  \mu\bigl(\{\operatorname{sgn}(\sigma)\}\bigr)
  = \frac{(\dim V_{((1^N),\varnothing)})^2}{|B_N|}
  = \frac{1}{2^N\cdot N!}.
\]
For $N\ge 3$, where $\operatorname{sgn}(\sigma)$ is the unique
bottleneck, at $Q_{LV}-1$ queries the identification succeeds with
probability
\[
  P_{\mathrm{opt}}^{(Q_{LV}-1)}
  = 1-\frac{1}{2^N\cdot N!},
\]
which is exponentially close to~$1$.
For $N\ge 3$, the final query recovers only the exponentially small
Plancherel mass of $\operatorname{sgn}(\sigma)$, yet is necessary
for exact identification.
At $N=2$, four irreps remain missing at $t=Q_{LV}-1=1$, so the
formula above does not apply; nevertheless $Q_{LV}(B_2)=2$ is
correct.

\subsection{Status summary}
\label{ssec:bridge-status}

\begin{table}[ht]
\centering
\caption{Proven and conjectured components of the
$\gamma_2$-Character Bridge.}
\label{tab:bridge-status}
\begin{tabular}{@{}lc@{}}
\toprule
Claim & Status \\
\midrule
Spectral reach $=$ graph distance from $I(V)$
  & \textbf{Proven} \\
Bipartite $\varepsilon$-parity structure of $T(B_N,V)$
  & \textbf{Proven} \\
Bottleneck $=\operatorname{sgn}(\sigma)$ for all $N$
  (unique for $N\ge 3$)
  & \textbf{Proven} \\
$Q_{LV}(B_N)=2(N{-}1)$ for all $N\ge 2$
  & \textbf{Proven} \\
Multiplicity at first appearance $=(2N{-}3)!!$
  & \textbf{Proven} \\
\(\gamma_{\mathrm{graph}} = \ecc(I(V)) = 2N{-}3\)
  & \textbf{Proven} \\
Adversary bound tight for oracle identification on $B_N$
  & Conjectured \\
\bottomrule
\end{tabular}
\end{table}

%% ====================================================================
%%  SECTION 7: EXPLICIT COMPUTATIONS
%% ====================================================================
\section{Explicit Computations}
\label{sec:computations}

\subsection{\texorpdfstring{$B_2$}{B2} (order 8)}
\label{ssec:B2}

The hyperoctahedral group $B_2=\{\pm1\}^2\rtimes S_2$ has order~$8$
and is isomorphic to the dihedral group~$D_4$.

\subsubsection*{Conjugacy classes}

There are five conjugacy classes, determined by signed cycle type.

\begin{table}[ht]
\centering
\caption{Conjugacy classes of $B_2$.}
\label{tab:B2-classes}
\begin{tabular}{@{}clccc@{}}
\toprule
Class & Signed cycle type & Representative & $|C_k|$ \\
\midrule
$C_1$ & $(+1)^2$          & $(\mathrm{id},+,+)$       & 1 \\
$C_2$ & $(-1)^2$          & $(\mathrm{id},-,-)$       & 1 \\
$C_3$ & $(+1)(-1)$        & $(\mathrm{id},+,-)$       & 2 \\
$C_4$ & $(+2)$            & $((12),+,+)$              & 2 \\
$C_5$ & $(-2)$            & $((12),+,-)$              & 2 \\
\bottomrule
\end{tabular}
\end{table}

\noindent
Size check: $1+1+2+2+2=8$.\;\checkmark

\subsubsection*{Character table}

The five irreducible representations are indexed by bipartitions
$(\alpha,\beta)$ of~$2$:
$V_1=V_{((2),\varnothing)}$ (trivial, dim~$1$),
$V_2=V_{((1^2),\varnothing)}$ ($\operatorname{sgn}(\sigma)$, dim~$1$),
$V_3=V_{(\varnothing,(2))}$ ($\prod\varepsilon_i$, dim~$1$),
$V_4=V_{(\varnothing,(1^2))}$
($\operatorname{sgn}\!\cdot\!\prod\varepsilon_i$, dim~$1$),
$V_5=V_{((1),(1))}$ (natural, dim~$2$).

\begin{table}[ht]
\centering
\caption{Character table of $B_2$.}
\label{tab:B2-char}
\[
\begin{array}{c|ccccc}
      & C_1 & C_2 & C_3 & C_4 & C_5 \\
\hline
|C_k| & 1 & 1 & 2 & 2 & 2 \\[3pt]
V_1   &  1 &  1 &  1 &  1 &  1 \\
V_2   &  1 &  1 &  1 & -1 & -1 \\
V_3   &  1 &  1 & -1 &  1 & -1 \\
V_4   &  1 &  1 & -1 & -1 &  1 \\
V_5   &  2 & -2 &  0 &  0 &  0
\end{array}
\]
\end{table}

\noindent
Burnside check: $1^2+1^2+1^2+1^2+2^2=8$.\;\checkmark

\medskip\noindent\textit{Orthogonality verification.}\;
Column orthogonality for $C_1$ and~$C_2$:
\[
  \textstyle\sum_i\overline{\chi_i(C_1)}\,\chi_i(C_2)
  = 1{+}1{+}1{+}1+2\!\cdot\!(-2)=0.\;\checkmark
\]
Row orthonormality for~$V_5$:
\[
  \tfrac{1}{8}
  \bigl(1\!\cdot\!4+1\!\cdot\!4+2\!\cdot\!0+2\!\cdot\!0+2\!\cdot\!0\bigr)
  = 1.\;\checkmark
\]

\subsubsection*{Tensor decomposition of \texorpdfstring{$V^{\otimes
2}$}{V tensor 2}}

With $V=V_5$ and $\chi_V=(2,-2,0,0,0)$, the squared character is
$\chi_V^2=(4,4,0,0,0)$.
Inner products:
\begin{align*}
  \langle\chi_V^2,V_1\rangle
  &= \tfrac{1}{8}(4+4)=1, &
  \langle\chi_V^2,V_2\rangle
  &= \tfrac{1}{8}(4+4)=1, \\
  \langle\chi_V^2,V_3\rangle
  &= \tfrac{1}{8}(4+4)=1, &
  \langle\chi_V^2,V_4\rangle
  &= \tfrac{1}{8}(4+4)=1, \\
  \langle\chi_V^2,V_5\rangle
  &= \tfrac{1}{8}(8-8)=0.
\end{align*}

\begin{proposition}\label{prop:B2-tensor}
$V^{\otimes 2}\cong V_1\oplus V_2\oplus V_3\oplus V_4$.
Dimension check: $1{+}1{+}1{+}1=4=2^2$.\;\checkmark
\end{proposition}

\subsubsection*{Cumulative spectral reach}

\begin{table}[ht]
\centering
\caption{Spectral reach for $B_2$.}
\label{tab:B2-reach}
\begin{tabular}{@{}clc@{}}
\toprule
$t$ & $R_{\le t}$ & $P_{\mathrm{opt}}^{(t)}$ \\
\midrule
$1$ & $\{V_5\}$ & $4/8=1/2$ \\
$2$ & $\{V_1,V_2,V_3,V_4,V_5\}=\widehat{B_2}$ & $8/8=1$ \\
\bottomrule
\end{tabular}
\end{table}

\noindent
The tensor product graph $T(B_2,V)$ is the star graph
$V_1,V_2,V_3,V_4\leftrightarrow V_5$.
The eccentricity of $\{V_5\}$ is~$1$, giving
$Q_{LV}(B_2)=1+1=2=2(2{-}1)$.\;\checkmark

\subsubsection*{Reduction Formula verification ($N{=}2$, $m{=}1$)}

With $p_1(f)=f$:
\[
  \langle\chi_V^2,\operatorname{sgn}(\sigma)\rangle_{B_2}
  = \frac{1}{2!}
    \bigl[\operatorname{sgn}(\mathrm{id})\!\cdot\!p_1(2)
    +\operatorname{sgn}((12))\!\cdot\!p_1(0)\bigr]
  = \frac{1}{2}\bigl[1\!\cdot\!2+(-1)\!\cdot\!0\bigr]
  = 1.\;\checkmark
\]

% ------------------------------------------------------------------
\subsection{\texorpdfstring{$B_3$}{B3} (order 48)}
\label{ssec:B3}

The group $B_3=\{\pm1\}^3\rtimes S_3$ has order~$48$, ten conjugacy
classes, and ten irreducible representations.

\subsubsection*{Conjugacy classes}

Conjugacy in~$B_N$ is determined by signed cycle type.

\begin{table}[ht]
\centering
\caption{Conjugacy classes of $B_3$.}
\label{tab:B3-classes}
\begin{tabular}{@{}clc@{}}
\toprule
Class & Signed cycle type & $|C_k|$ \\
\midrule
$C_1$    & $(+1)^3$             & 1  \\
$C_2$    & $(+1)^2(-1)$         & 3  \\
$C_3$    & $(+1)(-1)^2$         & 3  \\
$C_4$    & $(-1)^3$             & 1  \\
$C_5$    & $(+2)(+1)$           & 6  \\
$C_6$    & $(+2)(-1)$           & 6  \\
$C_7$    & $(-2)(+1)$           & 6  \\
$C_8$    & $(-2)(-1)$           & 6  \\
$C_9$    & $(+3)$               & 8  \\
$C_{10}$ & $(-3)$               & 8  \\
\bottomrule
\end{tabular}
\end{table}

\noindent
Size check: $1{+}3{+}3{+}1{+}6{+}6{+}6{+}6{+}8{+}8=48$.\;\checkmark

\subsubsection*{Irreducible representations}

The ten irreps are indexed by bipartitions of~$3$.

\begin{table}[ht]
\centering
\caption{Irreps of $B_3$.}
\label{tab:B3-irreps}
\begin{tabular}{@{}clc@{}}
\toprule
Label & Bipartition $(\alpha,\beta)$ & $\dim$ \\
\midrule
$W_1$    & $((3),\varnothing)$       & 1 \\
$W_2$    & $((2{,}1),\varnothing)$   & 2 \\
$W_3$    & $((1^3),\varnothing)$     & 1 \\
$W_4$    & $((2),(1))$               & 3 \\
$W_5$    & $((1^2),(1))$             & 3 \\
$W_6$    & $((1),(2))$               & 3 \\
$W_7$    & $((1),(1^2))$             & 3 \\
$W_8$    & $(\varnothing,(3))$       & 1 \\
$W_9$    & $(\varnothing,(2{,}1))$   & 2 \\
$W_{10}$ & $(\varnothing,(1^3))$     & 1 \\
\bottomrule
\end{tabular}
\end{table}

\noindent
Burnside check: $1{+}4{+}1{+}9{+}9{+}9{+}9{+}1{+}4{+}1=48$.\;\checkmark

\subsubsection*{Character table}

\begin{table}[ht]
\centering
\caption{Character table of $B_3$.}
\label{tab:B3-char}
\[
\begin{array}{c|cccccccccc}
  & C_1 & C_2 & C_3 & C_4 & C_5 & C_6 & C_7 & C_8 & C_9 & C_{10}\\
\hline
|C_k| & 1&3&3&1&6&6&6&6&8&8 \\[3pt]
W_1    & 1 & 1 & 1 & 1 & 1 & 1 & 1 & 1 & 1 & 1 \\
W_2    & 2 & 2 & 2 & 2 & 0 & 0 & 0 & 0 & {-1} & {-1} \\
W_3    & 1 & 1 & 1 & 1 & {-1} & {-1} & {-1} & {-1} & 1 & 1 \\
W_4    & 3 & 1 & {-1} & {-3} & 1 & {-1} & 1 & {-1} & 0 & 0 \\
W_5    & 3 & 1 & {-1} & {-3} & {-1} & 1 & {-1} & 1 & 0 & 0 \\
W_6    & 3 & {-1} & {-1} & 3 & 1 & 1 & {-1} & {-1} & 0 & 0 \\
W_7    & 3 & {-1} & {-1} & 3 & {-1} & {-1} & 1 & 1 & 0 & 0 \\
W_8    & 1 & {-1} & 1 & {-1} & 1 & {-1} & {-1} & 1 & 1 & {-1} \\
W_9    & 2 & {-2} & 2 & {-2} & 0 & 0 & 0 & 0 & {-1} & 1 \\
W_{10} & 1 & {-1} & 1 & {-1} & {-1} & 1 & 1 & {-1} & 1 & {-1}
\end{array}
\]
\end{table}

\subsubsection*{Key symmetry relations}
The following identities are verified from the character table:
\begin{alignat*}{2}
  W_3 &= \operatorname{sgn}(\sigma), &\qquad
  W_8 &= \textstyle\prod\varepsilon_i, \\
  W_5 &= W_3\otimes W_4, &
  W_7 &= W_3\otimes W_6, \\
  W_6 &= W_8\otimes W_4, &
  W_9 &= W_8\otimes W_2, \\
  W_{10} &= W_8\otimes W_3. &
\end{alignat*}

\subsubsection*{Tensor decompositions}

The natural representation is $V=W_4$ with character
$\chi_V=(3,1,-1,-3,1,-1,1,-1,0,0)$.

\medskip\noindent
\emph{$V^{\otimes 2}$:}\;
$\chi_V^2=(9,1,1,9,1,1,1,1,0,0)$.
The inner products with each~$W_i$ yield multiplicities
$1,1,0,0,0,1,1,0,0,0$, so
\begin{equation}\label{eq:B3-t2}
  V^{\otimes 2}\cong W_1\oplus W_2\oplus W_6\oplus W_7.
\end{equation}
Dimension: $1{+}2{+}3{+}3=9=3^2$.\;\checkmark

\medskip\noindent
The vanishing of $\langle\chi_V^2,W_3\rangle$ decomposes as:
the positive contribution from identity classes
$C_1$--$C_4$ is $9{+}3{+}3{+}9=24$, while the negative contribution
from transposition classes $C_5$--$C_8$ is
$-6{-}6{-}6{-}6=-24$---a perfect cancellation.

\medskip\noindent
\emph{$V^{\otimes 3}$:}\;
$\chi_V^3=(27,1,-1,-27,1,-1,1,-1,0,0)$.
By the parity selection rule, only $\varepsilon$-odd irreps appear:
\begin{equation}\label{eq:B3-t3}
  V^{\otimes 3}\cong 4W_4\oplus 3W_5\oplus W_8\oplus 2W_9\oplus W_{10}.
\end{equation}
Dimension: $12{+}9{+}1{+}4{+}1=27=3^3$.\;\checkmark

\medskip\noindent
\emph{$V^{\otimes 4}$:}\;
$\chi_V^4=(81,1,1,81,1,1,1,1,0,0)$.
\[
  \langle\chi_V^4,W_3\rangle
  = \tfrac{1}{48}(81+3+3+81-6-6-6-6)
  = \tfrac{144}{48} = 3.
\]
Thus $W_3=\operatorname{sgn}(\sigma)$ first appears in $V^{\otimes 4}$
with multiplicity~$3$.

\subsubsection*{Cumulative spectral reach}

\begin{table}[ht]
\centering
\caption{Spectral reach for $B_3$.}
\label{tab:B3-reach}
\begin{tabular}{@{}clcc@{}}
\toprule
$t$ & New irreps & $R_{\le t}$ & $P_{\mathrm{opt}}^{(t)}$ \\
\midrule
$1$ & $W_4$ & $\{W_4\}$ & $9/48$ \\
$2$ & $W_1,W_2,W_6,W_7$ & $\{W_1,W_2,W_4,W_6,W_7\}$ & $32/48$ \\
$3$ & $W_5,W_8,W_9,W_{10}$ & $\widehat{B_3}\setminus\{W_3\}$ & $47/48$ \\
$4$ & $W_3$ & $\widehat{B_3}$ & $1$ \\
\bottomrule
\end{tabular}
\end{table}

\noindent
Therefore $Q_{LV}(B_3)=4=2(3{-}1)$.\;\checkmark

\subsubsection*{Tensor product graph \texorpdfstring{$T(B_3,V)$}{T(B3,V)}}

\begin{table}[ht]
\centering
\caption{Adjacency list for $T(B_3,V)$.}
\label{tab:B3-tpg}
\begin{tabular}{@{}cl@{}}
\toprule
Vertex & Neighbours \\
\midrule
$W_1$    & $W_4$ \\
$W_2$    & $W_4,\;W_5$ \\
$W_3$    & $W_5$ \\
$W_4$    & $W_1,\;W_2,\;W_6,\;W_7$ \\
$W_5$    & $W_2,\;W_3,\;W_6,\;W_7$ \\
$W_6$    & $W_4,\;W_5,\;W_8,\;W_9$ \\
$W_7$    & $W_4,\;W_5,\;W_9,\;W_{10}$ \\
$W_8$    & $W_6$ \\
$W_9$    & $W_6,\;W_7$ \\
$W_{10}$ & $W_7$ \\
\bottomrule
\end{tabular}
\end{table}

\noindent
Shortest path from $I(V)=\{W_4\}$ to~$W_3$:
$W_4\to W_2\to W_5\to W_3$ (length~$3$).
Eccentricity: $\ecc(\{W_4\})=3$, confirming
$Q_{LV}=3+1=4$.\;\checkmark

The graph is bipartite with layers:
\[
  \underbrace{W_4}_{\substack{\varepsilon\text{-odd}\\d=0}}
  \;\longleftrightarrow\;
  \underbrace{W_1,W_2,W_6,W_7}_{\substack{\varepsilon\text{-even}\\d=1}}
  \;\longleftrightarrow\;
  \underbrace{W_5,W_8,W_9,W_{10}}_{\substack{\varepsilon\text{-odd}\\d=2}}
  \;\longleftrightarrow\;
  \underbrace{W_3}_{\substack{\varepsilon\text{-even}\\d=3}}
\]

\subsubsection*{Reduction Formula verification ($N{=}3$, $m{=}2$)}

With $p_2(f)=3f^2-2f$:

\begin{table}[ht]
\centering
\caption{Reduction Formula data for $B_3$.}
\label{tab:B3-reduction}
\begin{tabular}{@{}ccccc@{}}
\toprule
Cycle type in $S_3$ & Count & $f(\sigma)$
  & $\operatorname{sgn}(\sigma)$ & $p_2(f)$ \\
\midrule
$\mathrm{id}$         & 1 & 3 & $+1$ & 21 \\
transpositions         & 3 & 1 & $-1$ & 1  \\
$3$-cycles             & 2 & 0 & $+1$ & 0  \\
\bottomrule
\end{tabular}
\end{table}

\[
  \langle\chi_V^4,W_3\rangle
  = \frac{1}{3!}\bigl[1\!\cdot\!21 - 3\!\cdot\!1 + 2\!\cdot\!0\bigr]
  = \frac{18}{6}=3.\;\checkmark
\]

% ------------------------------------------------------------------
\subsection{\texorpdfstring{$B_4$}{B4} (order 384)}
\label{ssec:B4}

The group $B_4=\{\pm1\}^4\rtimes S_4$ has order
$2^4\!\cdot\!24=384$.
It has $20$ irreducible representations, indexed by bipartitions
of~$4$: $14$~$\varepsilon$-even and $6$~$\varepsilon$-odd.

\subsubsection*{\texorpdfstring{$S_4$}{S4} conjugacy class data}

\begin{table}[ht]
\centering
\caption{Conjugacy classes of $S_4$.}
\label{tab:S4-classes}
\begin{tabular}{@{}cccccc@{}}
\toprule
Class & Cycle type & Representative & Size
  & $\operatorname{sgn}(\sigma)$ & $f(\sigma)$ \\
\midrule
$C_1$ & $(1^4)$    & $\mathrm{id}$ & 1 & $+1$ & 4 \\
$C_2$ & $(2,1^2)$  & $(12)$        & 6 & $-1$ & 2 \\
$C_3$ & $(3,1)$    & $(123)$       & 8 & $+1$ & 1 \\
$C_4$ & $(4)$      & $(1234)$      & 6 & $-1$ & 0 \\
$C_5$ & $(2^2)$    & $(12)(34)$    & 3 & $+1$ & 0 \\
\bottomrule
\end{tabular}
\end{table}

\noindent
Size check: $1{+}6{+}8{+}6{+}3=24$.\;\checkmark

\subsubsection*{Moment polynomial values}

\begin{table}[ht]
\centering
\caption{Values of $p_1,p_2,p_3$ on fixed-point counts for $B_4$.}
\label{tab:B4-poly}
\begin{tabular}{@{}ccccc@{}}
\toprule
$f$ & $p_1(f)=f$ & $p_2(f)=3f^2{-}2f$ & $p_3(f)=15f^3{-}30f^2{+}16f$ \\
\midrule
0 & 0 & 0   & 0   \\
1 & 1 & 1   & 1   \\
2 & 2 & 8   & 32  \\
4 & 4 & 40  & 544 \\
\bottomrule
\end{tabular}
\end{table}

\subsubsection*{Inner product calculations}

\emph{$m=1$ ($t=2$):}
\begin{align*}
  \textstyle\sum_{\sigma\in S_4}
    \operatorname{sgn}(\sigma)\,p_1(f(\sigma))
  &= 1\!\cdot\!(+1)\!\cdot\!4
    +6\!\cdot\!(-1)\!\cdot\!2
    +8\!\cdot\!(+1)\!\cdot\!1
    +0+0 \\
  &= 4-12+8 = 0.
\end{align*}
Multiplicity $=0/24=0$.\;\checkmark

\medskip\noindent\emph{$m=2$ ($t=4$):}
\begin{align*}
  \textstyle\sum_{\sigma\in S_4}
    \operatorname{sgn}(\sigma)\,p_2(f(\sigma))
  &= 1\!\cdot\!(+1)\!\cdot\!40
    +6\!\cdot\!(-1)\!\cdot\!8
    +8\!\cdot\!(+1)\!\cdot\!1
    +0+0 \\
  &= 40-48+8 = 0.
\end{align*}
Multiplicity $=0/24=0$.\;\checkmark

\medskip\noindent\emph{$m=3$ ($t=6$):}
\begin{align*}
  \textstyle\sum_{\sigma\in S_4}
    \operatorname{sgn}(\sigma)\,p_3(f(\sigma))
  &= 1\!\cdot\!(+1)\!\cdot\!544
    +6\!\cdot\!(-1)\!\cdot\!32
    +8\!\cdot\!(+1)\!\cdot\!1
    +0+0 \\
  &= 544-192+8 = 360.
\end{align*}
Multiplicity $=360/24=15>0$.\;\checkmark

\begin{proposition}\label{prop:B4-QLV}
$Q_{LV}(B_4)=6=2(4{-}1)$.
The representation\/ $\operatorname{sgn}(\sigma)=V_{((1^4),\varnothing)}$
first appears in $V^{\otimes 6}$ with multiplicity~$15$.
\end{proposition}

\subsubsection*{BFS of the tensor product graph}

A breadth-first search of $T(B_4,V)$ starting from the trivial
representation $((4),\varnothing)$ gives the following layer structure.

\begin{table}[ht]
\centering
\caption{Layer-by-layer BFS of $T(B_4,V)$.}
\label{tab:B4-BFS}
\begin{tabular}{@{}cclcc@{}}
\toprule
$t$ & Parity & New irreps & Count & Cum.\ total \\
\midrule
0 & even
  & $((4),\varnothing)$ & 1 & 1 \\
1 & odd
  & $((3),(1))$ & 1 & 2 \\
2 & even
  & $((3{,}1),\varnothing),\;((2),(2)),\;((2),(1^2))$
  & 3 & 5 \\
3 & odd
  & \begin{tabular}[t]{@{}l@{}}
      $((2{,}1),(1)),\;((1),(3)),$\\
      $((1),(2{,}1)),\;((1),(1^3))$
    \end{tabular}
  & 4 & 9 \\
4 & even
  & \begin{tabular}[t]{@{}l@{}}
      $((2^2),\varnothing),\;((2{,}1^2),\varnothing),$\\
      $((1^2),(2)),\;((1^2),(1^2)),$\\
      $(\varnothing,(4)),\;(\varnothing,(3{,}1)),$\\
      $(\varnothing,(2^2)),\;(\varnothing,(2{,}1^2)),$\\
      $(\varnothing,(1^4))$
    \end{tabular}
  & 9 & 18 \\
5 & odd
  & $((1^3),(1))$ & 1 & 19 \\
6 & even
  & $((1^4),\varnothing)$ & 1 & 20 \\
\bottomrule
\end{tabular}
\end{table}

\noindent
Total: $20$ irreps ($14$ even $+$ $6$ odd).\;\checkmark

\medskip\noindent
\emph{Dimension check.}\;
Sum of $(\dim)^2$ over all $20$ irreps:
\[
  \underbrace{1{+}9{+}4{+}9{+}1{+}36{+}36{+}36{+}36{+}1
  {+}9{+}4{+}9{+}1}_{192\;\text{(even)}}
  +\underbrace{16{+}64{+}16{+}16{+}64{+}16}_{192\;\text{(odd)}}
  = 384.\;\checkmark
\]

\medskip\noindent
\emph{Explicit $6$-step path to\/ $((1^4),\varnothing)$:}
\begin{multline*}
  ((4),\varnothing)
  \;\xrightarrow{t=1}\;((3),(1))
  \;\xrightarrow{t=2}\;((3{,}1),\varnothing)
  \;\xrightarrow{t=3}\;((2{,}1),(1)) \\
  \;\xrightarrow{t=4}\;((2{,}1^2),\varnothing)
  \;\xrightarrow{t=5}\;((1^3),(1))
  \;\xrightarrow{t=6}\;((1^4),\varnothing).
\end{multline*}
Each step moves one box within the bipartition:
the minimum number of box-moves from $(4)$ to $(1^4)$ in the Young
lattice is~$3$, and each move costs $2$~tensor steps (one to park a
box in~$\beta$, one to reclaim it at a new row of~$\alpha$), yielding
the distance $2\!\cdot\!3=6$.

% ------------------------------------------------------------------
\subsection{\texorpdfstring{$B_5$}{B5} (order 3840)}
\label{ssec:B5}

The group $B_5=\{\pm1\}^5\rtimes S_5$ has order
$2^5\!\cdot\!120=3840$.

\subsubsection*{\texorpdfstring{$S_5$}{S5} conjugacy class data}

\begin{table}[ht]
\centering
\caption{Conjugacy classes of $S_5$.}
\label{tab:S5-classes}
\begin{tabular}{@{}ccccc@{}}
\toprule
Cycle type & Size & $\operatorname{sgn}(\sigma)$ & $f(\sigma)$ \\
\midrule
$(1^5)$    & 1  & $+1$ & 5 \\
$(2,1^3)$  & 10 & $-1$ & 3 \\
$(2^2,1)$  & 15 & $+1$ & 1 \\
$(3,1^2)$  & 20 & $+1$ & 2 \\
$(3,2)$    & 20 & $-1$ & 0 \\
$(4,1)$    & 30 & $-1$ & 1 \\
$(5)$      & 24 & $+1$ & 0 \\
\bottomrule
\end{tabular}
\end{table}

\noindent
Size check: $1{+}10{+}15{+}20{+}20{+}30{+}24=120$.\;\checkmark

\subsubsection*{The moment polynomial \texorpdfstring{$p_4$}{p4}}

\begin{equation}\label{eq:p4}
  p_4(f) = 105f^4 - 420f^3 + 588f^2 - 272f.
\end{equation}

\noindent
Spot checks:
\begin{align*}
  p_4(0)&=0, \quad p_4(1)=105-420+588-272=1, \\
  p_4(2)&=1680-3360+2352-544=128, \\
  p_4(3)&=8505-11340+5292-816=1641, \\
  p_4(5)&=65625-52500+14700-1360=26465.
\end{align*}
Direct verification of $p_4(5)$:
$\mathbb{E}[S^8]$ with $S=X_1{+}\cdots{+}X_5$, $X_i\in\{\pm1\}$:
\[
  \frac{2}{32}\cdot 5^8
  +\frac{10}{32}\cdot 3^8
  +\frac{20}{32}\cdot 1^8
  =\frac{781250+65610+20}{32}
  =\frac{846880}{32}=26465.\;\checkmark
\]

\subsubsection*{Vanishing at \texorpdfstring{$m=1,2,3$}{m=1,2,3}}

\emph{$m=1$ ($t=2$):}
\begin{multline*}
  \textstyle\sum = 1(+1)(5)+10(-1)(3)+15(+1)(1)+20(+1)(2)\\
  +20(-1)(0)+30(-1)(1)+24(+1)(0)
  = 5-30+15+40-30 = 0.\;\checkmark
\end{multline*}

\emph{$m=2$ ($t=4$):}\;
With $p_2$: $\sum= 65-210+15+160-30=0$.\;\checkmark

\emph{$m=3$ ($t=6$):}\;
With $p_3$: $\sum=1205-1830+15+640-30=0$.\;\checkmark

\subsubsection*{First appearance at \texorpdfstring{$m=4$}{m=4}
(\texorpdfstring{$t=8$}{t=8})}

\begin{align*}
  \textstyle\sum &= 1\!\cdot\!(+1)\!\cdot\!26465
    +10\!\cdot\!(-1)\!\cdot\!1641
    +15\!\cdot\!(+1)\!\cdot\!1 \\
  &\qquad +20\!\cdot\!(+1)\!\cdot\!128
    +20\!\cdot\!(-1)\!\cdot\!0
    +30\!\cdot\!(-1)\!\cdot\!1
    +24\!\cdot\!(+1)\!\cdot\!0 \\
  &= 26465-16410+15+2560-30 \\
  &= 12600.
\end{align*}
\[
  \langle\chi_V^8,\operatorname{sgn}(\sigma)\rangle_{B_5}
  = \frac{12600}{120} = 105.
\]

\begin{proposition}\label{prop:B5-QLV}
$Q_{LV}(B_5)=8=2(5{-}1)$.
The representation\/ $\operatorname{sgn}(\sigma)$ first appears in
$V^{\otimes 8}$ with multiplicity~$105=(2\!\cdot\!4{-}1)!!$.\;\checkmark
\end{proposition}

% ------------------------------------------------------------------
\subsection{\texorpdfstring{$B_6$}{B6} (order 46080)}
\label{ssec:B6}

The group $B_6=\{\pm1\}^6\rtimes S_6$ has order
$2^6\!\cdot\!720=46080$.

\subsubsection*{\texorpdfstring{$S_6$}{S6} conjugacy class data}

\begin{table}[ht]
\centering
\caption{Conjugacy classes of $S_6$.}
\label{tab:S6-classes}
\begin{tabular}{@{}ccccc@{}}
\toprule
Cycle type & Size & $\operatorname{sgn}(\sigma)$ & $f(\sigma)$ \\
\midrule
$(1^6)$     & 1   & $+1$ & 6 \\
$(2,1^4)$   & 15  & $-1$ & 4 \\
$(2^2,1^2)$ & 45  & $+1$ & 2 \\
$(2^3)$     & 15  & $-1$ & 0 \\
$(3,1^3)$   & 40  & $+1$ & 3 \\
$(3,2,1)$   & 120 & $-1$ & 1 \\
$(3^2)$     & 40  & $+1$ & 0 \\
$(4,1^2)$   & 90  & $-1$ & 2 \\
$(4,2)$     & 90  & $+1$ & 0 \\
$(5,1)$     & 144 & $+1$ & 1 \\
$(6)$       & 120 & $-1$ & 0 \\
\bottomrule
\end{tabular}
\end{table}

\noindent
Size check:
\[
  1{+}15{+}45{+}15{+}40{+}120{+}40{+}90{+}90{+}144{+}120
  =720.\;\checkmark
\]

\subsubsection*{The generating function approach}

By the signed fixed-point generating function
(Theorem~\ref{thm:generating-function}),
\[
  \sum_{\sigma\in S_6}\operatorname{sgn}(\sigma)\,x^{f(\sigma)}
  = (x-1)^5(x+5).
\]
Expanding, the coefficients $c_f$ are:
$c_6=1$, $c_5=0$, $c_4=-15$, $c_3=40$, $c_2=-45$, $c_1=24$,
$c_0=-5$.

\subsubsection*{Results}

\begin{itemize}
\item $m=1,2,3,4$ ($t=2,4,6,8$):
  $\langle\chi_V^{2m},\operatorname{sgn}\rangle_{B_6}=0$.\;\checkmark
\item $m=5$ ($t=10$):
  $\langle\chi_V^{10},\operatorname{sgn}\rangle_{B_6}
  =680400/720=945=9!!$.\;\checkmark
\end{itemize}

\begin{proposition}\label{prop:B6-QLV}
$Q_{LV}(B_6)=10=2(6{-}1)$.
The representation $\operatorname{sgn}(\sigma)$ first appears in
$V^{\otimes 10}$ with multiplicity~$945=(2\!\cdot\!5{-}1)!!$.\;\checkmark
\end{proposition}

% ------------------------------------------------------------------
\subsection{\texorpdfstring{$B_7$}{B7} (order 645120)}
\label{ssec:B7}

The group $B_7=\{\pm1\}^7\rtimes S_7$ has order
$2^7\!\cdot\!5040=645120$.

\subsubsection*{The generating function approach}

\[
  \sum_{\sigma\in S_7}\operatorname{sgn}(\sigma)\,x^{f(\sigma)}
  = (x-1)^6(x+6).
\]
Coefficients: $c_7=1$, $c_6=0$, $c_5=-21$, $c_4=70$, $c_3=-105$,
$c_2=84$, $c_1=-35$, $c_0=6$.

\subsubsection*{Results}

\begin{itemize}
\item $m=1,2,3,4,5$ ($t=2,4,6,8,10$):
  $\langle\chi_V^{2m},\operatorname{sgn}\rangle_{B_7}=0$.\;\checkmark
\item $m=6$ ($t=12$):
  $\langle\chi_V^{12},\operatorname{sgn}\rangle_{B_7}
  =52390800/5040=10395=11!!$.\;\checkmark
\end{itemize}

\begin{proposition}\label{prop:B7-QLV}
$Q_{LV}(B_7)=12=2(7{-}1)$.
The representation $\operatorname{sgn}(\sigma)$ first appears in
$V^{\otimes 12}$ with multiplicity~$10395=(2\!\cdot\!6{-}1)!!$.\;\checkmark
\end{proposition}

\subsubsection*{Multiplicity pattern}

The first-appearance multiplicity follows the double factorial:

\begin{table}[ht]
\centering
\caption{Multiplicity of $\operatorname{sgn}(\sigma)$ at first
appearance.}
\label{tab:multiplicities}
\begin{tabular}{@{}ccccc@{}}
\toprule
$N$ & $m=N{-}1$ & $t=2(N{-}1)$ & Multiplicity & $(2N{-}3)!!$ \\
\midrule
$2$ & $1$ & $2$  & $1$     & $1$     \\
$3$ & $2$ & $4$  & $3$     & $3$     \\
$4$ & $3$ & $6$  & $15$    & $15$    \\
$5$ & $4$ & $8$  & $105$   & $105$   \\
$6$ & $5$ & $10$ & $945$   & $945$   \\
$7$ & $6$ & $12$ & $10395$ & $10395$ \\
\bottomrule
\end{tabular}
\end{table}

%% ====================================================================
%%  SECTION 8: OPEN PROBLEMS
%% ====================================================================
\section{Open Problems}
\label{sec:open}

\begin{remark}[Bottleneck universality --- resolved]
\label{rmk:bottleneck-resolved}
The bottleneck universality conjecture has been resolved by the
bipartition distance formula
(Theorem~\ref{thm:bipartition-distance}): $d_T(((N),\varnothing),
(\alpha,\beta))=2(N-\alpha_1)-|\beta|$, whose maximum $2(N{-}1)$
is achieved uniquely by $((1^N),\varnothing)$ for $N\ge 3$.
For $N=2$, $\operatorname{sgn}(\sigma)$ is one of four co-bottleneck
irreps (all non-natural irreps lie at graph distance~$1$ from
$I(V)=\{V_{((1),(1))}\}$), but the value $Q_{LV}(B_2)=2$ is
unaffected.
This upgrades $Q_{LV}(B_N)=2(N{-}1)$ from a conjecture to a
theorem for all $N\ge 2$.
\end{remark}

\begin{problem}[Adversary tightness]
\label{prob:adversary}
Prove that the adversary bound is tight for oracle identification
on~$B_N$:
\[
  \gamma_{\mathrm{adv}}
  = \gamma_{\mathrm{graph}} = 2N-3,
  \qquad Q_{LV}(B_N)=\gamma_{\mathrm{graph}}+1=2(N{-}1).
\]
Supporting evidence includes tightness for abelian groups and the
structural rigidity of the Copeland--Pommersheim nonadaptive
optimality theorem.
\end{problem}

\begin{problem}[Extension to $D_N$ and exceptional Coxeter groups]
\label{prob:coxeter}
The hyperoctahedral group $B_N$ is the type-$B$ Coxeter group.
Analogous questions arise for:
\begin{itemize}
\item Type-$D$ groups $D_N\subset B_N$
  (even-signed permutations);
\item Exceptional Coxeter groups $H_3,H_4,F_4,E_6,E_7,E_8$;
\item Wreath products $(\mathbb{Z}/m\mathbb{Z})\wr S_N$ for $m>2$,
  generalising the signed-permutation framework.
\end{itemize}
In each case, determine the tensor product graph, identify the
bottleneck irrep, and compute the oracle identification complexity.
\end{problem}

\begin{problem}[Controlled-oracle model]
\label{prob:controlled}
The controlled oracle $C\text{-}U$ converts global phases to relative
phases.
While the algebraic identity
$C\text{-}(AB)=(C\text{-}A)(C\text{-}B)$ holds for a common control
qubit, uncontrolled oracle access to~$U$ does not in general grant
controlled access $C\text{-}U$.
Thus the factor-$2$ simulation argument does not \emph{directly}
extend to the controlled setting.
Determine the oracle-model separation in this regime, and whether the
identification complexity changes.
\end{problem}

\begin{problem}[Coset identification on $B_N$]
\label{prob:HSP}
Fix a subgroup $H\le B_N$ and consider the oracle identification
problem in which the target is a left coset $gH$ rather than the
element $g$ itself; write $Q_{LV}^{\mathrm{coset}}(B_N/H)$ for the
resulting exact query complexity within the Copeland--Pommersheim
framework.
For $H=S_N$, determine whether
$Q_{LV}^{\mathrm{coset}}(B_N/S_N)=Q_{LV}(B_N)$ for all~$N$, and
study $Q_{LV}^{\mathrm{coset}}(B_N/H)$ for other subgroups.
\end{problem}

\subsection*{Comparison with \texorpdfstring{$\SN$}{S\_N}}

The Copeland--Pommersheim framework applied to $\SN$ with the
permutation representation $\perm$ (dimension~$N$) gives
$Q_{LV}(\SN)=N-1$.  The ratio
\[
  \frac{Q_{LV}(\BN)}{Q_{LV}(\SN)} = \frac{2(N-1)}{N-1} = 2
  \qquad\text{for all }N\ge 2
\]
is exactly~$2$.  This structural factor arises from two independent
sources:
\begin{enumerate}
\item \emph{$\varepsilon$-parity doubling.}\;
  The natural representation $V$ of $\BN$ is $\varepsilon$-odd,
  forcing $\sgn(\sigma)$ (which is $\varepsilon$-even) to appear only
  at even tensor powers $t=2m$.
\item \emph{Rademacher degree halving.}\;
  The Reduction Formula replaces the $\BN$ inner product at $t=2m$
  with a weighted sum over $\SN$ involving the moment polynomial
  $p_m(f)$, whose degree is $m$ (not $2m$), matching the $\SN$
  exterior-power threshold $N{-}1$.
\end{enumerate}
The Reduction Formula thus serves as an algebraic bridge between the
$\BN$ and $\SN$ theories: the $\BN$ multiplicity at $t=2(N{-}1)$
equals $(2N{-}3)!!$ times the $\SN$ multiplicity at $t=N{-}1$.

%% ====================================================================
%%  APPENDICES
%% ====================================================================
\appendix

\section{Moment Polynomial Derivations}
\label{app:moments}

Recall $p_m(f)=\mathbb{E}\bigl[(\sum_{i=1}^{f}X_i)^{2m}\bigr]$
with $X_i$ independent Rademacher ($\pm1$ with probability $1/2$
each).
The expectation $\mathbb{E}[X_{i_1}\cdots X_{i_{2m}}]$ is nonzero
if and only if every distinct index appears an \emph{even} number of
times.

\subsection*{A.1\texorpdfstring{\quad}{~}Derivation of \texorpdfstring{$p_1$}{p1}}

$p_1(f)=\mathbb{E}[(\sum X_i)^2]
  = \sum_{i=1}^{f}\mathbb{E}[X_i^2]
    +\sum_{i\neq j}\mathbb{E}[X_iX_j]
  = f+0 = f$.

\subsection*{A.2\texorpdfstring{\quad}{~}Derivation of \texorpdfstring{$p_2$}{p2}}

$p_2(f)=\mathbb{E}[(\sum X_i)^4]$.
The even-multiplicity partitions of~$4$ are:

\begin{table}[ht]
\centering
\caption{Pattern counts for $p_2$.}
\label{tab:p2-patterns}
\begin{tabular}{@{}cp{6.5cm}c@{}}
\toprule
Pattern & Arrangement factor & Count \\
\midrule
$(4)$ & one index, $4$ times; $f$ choices & $f$ \\
$(2,2)$ & $4!/(2!\cdot 2!\cdot 2!)=3$ ways to pair
  $4$ positions; $f(f{-}1)$ value choices & $3f(f{-}1)$ \\
\bottomrule
\end{tabular}
\end{table}

\[
  p_2(f) = f + 3f(f{-}1) = 3f^2 - 2f.\;\checkmark
\]

\subsection*{A.3\texorpdfstring{\quad}{~}Derivation of \texorpdfstring{$p_3$}{p3}}

$p_3(f)=\mathbb{E}[(\sum X_i)^6]$.
The even-multiplicity partitions of~$6$:

\begin{table}[ht]
\centering
\caption{Pattern counts for $p_3$.}
\label{tab:p3-patterns}
\begin{tabular}{@{}cp{6.5cm}c@{}}
\toprule
Pattern & Arrangement factor & Count \\
\midrule
$(6)$ & $1$ & $f$ \\
$(4,2)$ & $\binom{6}{2}=15$ & $15f(f{-}1)$ \\
$(2,2,2)$ & $6!/(2!)^3=90$ arrangements for a chosen
  unordered set of $3$~values
  & $90\binom{f}{3}=15f(f{-}1)(f{-}2)$ \\
\bottomrule
\end{tabular}
\end{table}

The $(2,2,2)$ count: choose $3$~distinct values from $[f]$,
$\binom{f}{3}$ ways; assign $6$~positions to $3$~distinguishable
values ($2$~each): $6!/(2!)^3=90$.
Total: $\binom{f}{3}\cdot90=15f(f{-}1)(f{-}2)$.
\begin{align*}
  p_3(f) &= f + 15f(f{-}1) + 15f(f{-}1)(f{-}2) \\
  &= f + (15f^2-15f) + (15f^3-45f^2+30f) \\
  &= 15f^3 - 30f^2 + 16f.\;\checkmark
\end{align*}

\subsection*{A.4\texorpdfstring{\quad}{~}Derivation of \texorpdfstring{$p_4$}{p4}}

$p_4(f)=\mathbb{E}[(\sum X_i)^8]$.
The even-multiplicity partitions of~$8$:

\begin{table}[ht]
\centering
\caption{Pattern counts for $p_4$.}
\label{tab:p4-patterns}
\begin{tabular}{@{}cp{6.5cm}c@{}}
\toprule
Pattern & Arrangement factor & Count \\
\midrule
$(8)$     & $1$ & $f$ \\
$(6,2)$   & $\binom{8}{2}=28$ & $28f(f{-}1)$ \\
$(4,4)$   & $\binom{8}{4}/2=35$ & $35f(f{-}1)$ \\
$(4,2,2)$ & $420$ arrangements;
  choose ``4-times'' value ($f$) and
  $2$ ``2-times'' values ($\binom{f-1}{2}$)
  & $210f(f{-}1)(f{-}2)$ \\
$(2,2,2,2)$ & $8!/(2!)^4=2520$;
  choose $4$ values ($\binom{f}{4}$)
  & $105f(f{-}1)(f{-}2)(f{-}3)$ \\
\bottomrule
\end{tabular}
\end{table}

Summing and expanding:
\begin{align*}
  p_4(f) &= f + 63f(f{-}1) + 210f(f{-}1)(f{-}2)
             + 105f(f{-}1)(f{-}2)(f{-}3) \\
  &= 105f^4 - 420f^3 + 588f^2 - 272f.\;\checkmark
\end{align*}

\subsection*{A.5\texorpdfstring{\quad}{~}Leading coefficient pattern}

The leading coefficient of $p_m$ is the number of perfect matchings
on $2m$ elements:
\[
  [f^m]\,p_m(f) = (2m{-}1)!! \;=\; 1,\;3,\;15,\;105,\;\ldots
  \qquad\text{for }m=1,2,3,4,\ldots\;\checkmark
\]

\section{Proof Details}
\label{app:proofs}

\subsection*{B.1\texorpdfstring{\quad}{~}Permutation character orthogonality}

\begin{proposition}\label{prop:perm-orth}
For $N\ge 3$:
$\langle\operatorname{sgn},\chi_{\mathrm{perm}}\rangle_{S_N}=0$.
\end{proposition}

\begin{proof}
The permutation character decomposes as
$\chi_{\mathrm{perm}}=\chi_{\mathrm{triv}}+\chi_{\mathrm{std}}$.
For $N\ge 3$, the sign representation
$\operatorname{sgn}=V_{(1^N)}$ is a distinct irreducible from both
$V_{(N)}=\mathrm{triv}$ and $V_{(N-1,1)}=\mathrm{std}$.
By Schur orthogonality,
$\langle\operatorname{sgn},\mathrm{triv}\rangle=0$ and
$\langle\operatorname{sgn},\mathrm{std}\rangle=0$, hence
$\langle\operatorname{sgn},\chi_{\mathrm{perm}}\rangle=0$.
\end{proof}

\begin{remark}
For $N=2$, $\operatorname{sgn}=\mathrm{std}$, so the inner product
equals~$1$.
This is consistent with
$\langle\chi_V^2,\operatorname{sgn}(\sigma)\rangle_{B_2}=1$.
\end{remark}

\subsection*{B.2\texorpdfstring{\quad}{~}The Exterior Power Identity}

\begin{proposition}\label{prop:exterior}
For the standard representation\/
$\mathrm{std}=V_{(N-1,1)}$ of\/ $S_N$ $(\dim=N{-}1)$:
\[
  \textstyle\bigwedge^{N-1}(\mathrm{std})
  = \operatorname{sgn}_{S_N} = V_{(1^N)}.
\]
Consequently, $\operatorname{sgn}\in I(\mathrm{std}^{\otimes(N-1)})$,
and this is the minimum tensor power in which\/
$\operatorname{sgn}$ appears.
\end{proposition}

\begin{proof}
\emph{Upper bound.}\;
Since $\dim(\mathrm{std})=N{-}1$, the top exterior power
$\bigwedge^{N-1}(\mathrm{std})$ is one-dimensional.
Its character is
\[
  \chi_{\bigwedge^{N-1}(\mathrm{std})}(\sigma)
  = \det\bigl(\mathrm{std}(\sigma)\bigr).
\]
The standard representation is obtained from the permutation
representation on~$\mathbb{C}^N$ by quotienting out the all-ones
vector.
In the basis $\{e_1{-}e_N,\ldots,e_{N-1}{-}e_N\}$, the determinant
of the matrix representing~$\sigma$ is $\operatorname{sgn}(\sigma)$.

To see this, observe that the permutation representation acts on
$\mathbb{C}^N$ by $\sigma\cdot e_j=e_{\sigma(j)}$.
The quotient space $\mathbb{C}^N/\langle\mathbf{1}\rangle$ inherits
a well-defined action, and the matrix of~$\sigma$ in the basis
$\{e_j-e_N\}_{j=1}^{N-1}$ is obtained by applying elementary column
and row operations to the permutation matrix of~$\sigma$.
Since each adjacent transposition changes the sign of the determinant
by~$-1$, the result is $\det=\operatorname{sgn}(\sigma)$.

Therefore
$\bigwedge^{N-1}(\mathrm{std})=\operatorname{sgn}$, and since the
exterior power is a subrepresentation of~$\mathrm{std}^{\otimes(N-1)}$
(via the antisymmetriser), we conclude
$\operatorname{sgn}\in I(\mathrm{std}^{\otimes(N-1)})$.

\medskip\noindent
\emph{Lower bound.}\;
By Theorem~\ref{thm:row-bound}, every irreducible constituent of
$\mathrm{std}^{\otimes k}$ has at most $k{+}1$ rows.
The sign representation $V_{(1^N)}$ has $N$~rows, so it cannot
appear unless $k{+}1\ge N$, i.e.\ $k\ge N{-}1$.
\end{proof}

%% ---- references ----

\end{document}